\documentclass[10pt]{article}%
\usepackage{amsmath}
\usepackage{amsfonts}
\usepackage{mathrsfs}
\usepackage{amssymb,color}
\usepackage[linkcolor=black,anchorcolor=black,citecolor=black]{hyperref}
\usepackage{graphicx}
\numberwithin{equation}{section}
\usepackage[body={15.5cm,21cm}, top=3cm]{geometry}%
\usepackage[numbers,sort&compress]{natbib}
\providecommand{\U}[1]{\protect\rule{.1in}{.1in}}
\providecommand{\U}[1]{\protect \rule{.1in}{.1in}}
\newtheorem{theorem}{Theorem}[section]

\newtheorem{assumption}[theorem]{Assumption}
\newtheorem{definition}[theorem]{Definition}

\newtheorem{lemma}[theorem]{Lemma}

\newtheorem{proposition}[theorem]{Proposition}
\newtheorem{remark}[theorem]{Remark}

\newenvironment{proof}[1][Proof]{\noindent \textbf{#1.} }{\  \rule{0.5em}{0.5em}}

\def \E{\mathsf{E}}

\begin{document}
	\title{The Skorokhod problem with two nonlinear constraints}
	\author{ Hanwu Li\thanks{Research Center for Mathematics and Interdisciplinary Sciences, Shandong University, Qingdao 266237, Shandong, China. lihanwu@sdu.edu.cn.}
	\thanks{Frontiers Science Center for Nonlinear Expectations (Ministry of Education), Shandong University, Qingdao 266237, Shandong, China.} }
	\date{}
	\maketitle
	\begin{abstract}
	In this paper, we study the Skorokhod problem with two constraints, where the constraints are in a nonlinear fashion. We prove the existence and uniqueness of the solution and also provide the explicit construction for the solution. In addition, a number of properties of the solution, including continuity under uniform and $J_1$ metrics, comparison principle and non-anticipatory property, are investigated. 
	\end{abstract}
	
	\textbf{Key words}: Skorokhod problem, nonlinear reflecting boundaries, double reflecting boundaries, comparison principle
	
	\textbf{MSC-classification}: 60G05, 60G17

\section{Introduction}

	The Skorokhod problem is a convenient tool to study equations with reflecting boundary conditions. In 1961, Skorokhod \cite{Skorokhod1} originally constructed the solution of a stochastic differential equation (SDE for short) on the half-line $[0,\infty)$. A nondecreasing function is added in this equation to push the solution upward in a minimal way such that it satisfies to so-called Skorokhod condition. In \cite{CE} and  \cite{CEM}, the authors considered a deterministic version of the Skorokhod problem for continuous functions and for c\`{a}dl\`{a}g functions, respectively. A multidimensional extension of the Skorokhod problem was provided by Tanaka \cite{Tanaka}.
	
	Due to the wide applications of reflecting Brownian motions including statistical physics \cite{BN,SW'}, queueing theory \cite{MM}, control theory \cite{EK}, the Skorokhod problem with two reflecting boundaries, called also two-sided Skorokhod problem, has attracted a great deal of attention by many researchers. Roughly speaking, the Skorokhod problem with two reflecting boundaries $\alpha,\beta$ is to find a pair of functions $(X,K)$ such that $X_t=S_t+K_t\in[\alpha_t,\beta_t]$ for any $t\geq 0$ and $K$ satisfies some necessary conditions, where $S,\alpha,\beta$ ($\alpha<\beta$) are some given right-continuous functions with left limits. For simplicity, $(X,K)$ is called the solution to the Skorokhod problem on $[\alpha,\beta]$ for $S$. Kruk et al. \cite{KLR} presented an explicit formula to make a deterministic function stays in the interval $[0,a]$ (i.e., $\alpha,\beta$ are two constants) and studied the properties of the solutions. Then, Burdzy et al. \cite{BKR} considered the Skorokhod problem in a time-dependent interval. They obtained the existence and uniqueness of a solution to the so-called extended Skorokhod problem, which is a slight generalization of the Skorokhod problem. Under the assumption that $\inf_t(\beta_t-\alpha_t)>0$, the solution to the extended Skorokhod problem coincides with the one to the Skorokhod problem.  We refer the interested reader to \cite{HJO,S1,S2,SW10,SW} and the references therein for related works in this field.

It is worth pointing out that in the existing literatures, the solution of a Skorokhod problem with two reflecting boundaries is required to remain in a (time-dependent) interval. The objective of this paper is to study the Skorokhod problem with two reflecting boundaries behaving in a nonlinear way, that is, we need to make sure that two functions of the solution stay positive and negative, respectively. More precisely, let $S$ be a right-continuous function with left limits on $[0,\infty)$ taking values in $\mathbb{R}$. Given two functions $L,R:[0,\infty)\times\mathbb{R}\rightarrow \mathbb{R}$ with $L<R$, we need to find a pair of functions $(X,K)$, such that 
\begin{itemize}
\item[(i)] $X_t=S_t+K_t$;
\item[(ii)] $L(t,X_t)\leq 0\leq R(t,X_t)$;
\item[(iii)] $K_{0-}=0$ and $K$ has the decomposition $K=K^r-K^l$, where $K^r,K^l$ are nondecreasing functions satisfying
\begin{align*}
\int_0^\infty I_{\{L(s,X_s)<0\}}dK^l_s=0, \  \int_0^\infty I_{\{R(s,X_s)>0\}}dK^r_s=0.
\end{align*}
\end{itemize} 		
By choosing $L,R$ appropriately, the Skorokhod problem with two nonlinear reflecting boundaries may degenerate into the classical Skorokhod problem in \cite{Skorokhod1}, the Skorokhod map on $[0,a]$ in \cite{KLR} and the Skorokhod problem in a time-dependent interval in \cite{BKR}, \cite{S1}, \cite{S2}. The difficulty in solving the nonlinear problem is to construct the bounded variation function $K$ without precisely given obstacles $\alpha,\beta$. Recalling the results in \cite{BKR}, the second component of the solution to the Skorokhod problem on $[\alpha,\beta]$ for $S$ is given by
\begin{align*}
K_t=-\max\left(\left[(S_0-\beta_0)^+\wedge \inf_{u\in [0,t]}(S_u-\alpha_u)\right],\sup_{s\in[0,t]}\left[(S_s-\beta_s)\wedge \inf_{u\in[s,t]}(S_u-\alpha_u)\right]\right).
\end{align*}
Motivated by the construction for solutions to Skorokhod problems with two reflecting boundaries studied in \cite{KLR}, \cite{S1},  for the nonlinear case,  we first find the solutions $\Phi^S$, $\Psi^S$  satisfying the following nonlinear reflecting constraints, respectively
\begin{align*}
L(t,S_t+\Phi^S_t)=0, \ R(t,S_t+\Psi^S_t)=0, \ t\geq 0.
\end{align*}
We show that, $\Phi^S$ and $\Psi^S$ will take over the roles of $\beta-S$ and $\alpha-S$, respectively and the induced $K$ is the second component of the solution to the Skorokhod problem with two nonlinear reflecting boundaries. The explicit characterization for $K$ allows us to obtain the continuity property of the solution with respect to the input function $S$ and the nonlinear functions $L,R$. We also present some comparison theorems. Roughly speaking, the nondecreasing function $K^r$ aims to push the solution upward such that $R(t,X_t)\geq 0$ while the nondecreasing function $K^l$ tries to pull the solution downward to makes sure that $L(t,X_t)\leq 0$. It is natural to conjecture that the forces $K^r$ and $K^l$ will increase if $R$ becomes smaller or $L$ becomes larger. The Skorokhod problem with two nonlinear reflecting boundaries will be a building block for studying doubly mean reflected (backward) SDEs (see \cite{FS,FS2} for the linear reflecting case).

The paper is organized as follows. We first formulate the Skorokhod problem with two nonlinear reflecting boundaries in details and provide the existence and uniqueness result in Section 2. Then, in Section 3, we investigated the properties of solutions to Skorokhod problems, such as the comparison property and the continuity property.

\section{Skorokhod problem with two nonlinear reflecting boundaries}

\subsection{Basic notations and problem formulation}
Let $D[0,\infty)$ be the set of real-valued right-continuous functions having left limits (usually called c\`{a}dl\`{a}g functions). $I[0,\infty)$, $C[0,\infty)$, $BV[0,\infty)$ and $AC[0,\infty)$ denote the subset of $D[0,\infty)$ consisting of nondecreasing functions, continuous functions, functions of bounded variation and absolutely continuous functions, respectively. For any $K\in BV[0,\infty)$ and $t\geq 0$, $|K|_t$ is the total variation of $K$ on $[0,t]$.

\begin{definition}\label{def}
Let $S\in D[0,\infty)$, $L,R:[0,\infty)\times \mathbb{R}\rightarrow \mathbb{R}$ be two functions with $L\leq R$. A pair of functions $(X,K)\in D[0,\infty)\times BV[0,\infty)$ is called a solution of the Skorokhod problem for $S$ with nonlinear constraints $L,R$ ($(X,K)=\mathbb{SP}_L^R(S)$ for short) if 
\begin{itemize}
\item[(i)] $X_t=S_t+K_t$;
\item[(ii)] $L(t,X_t)\leq 0\leq R(t,X_t)$;
\item[(iii)] $K_{0-}=0$ and $K$ has the decomposition $K=K^r-K^l$, where $K^r,K^l$ are nondecreasing functions satisfying
\begin{align*}
\int_0^\infty I_{\{L(s,X_s)<0\}}dK^l_s=0, \  \int_0^\infty I_{\{R(s,X_s)>0\}}dK^r_s=0.
\end{align*}
\end{itemize}
\end{definition}

\begin{remark}\label{rem1}
(i)  The integration in (iii) of Definition \ref{def} is carried out including the initial time $0$. That is, if $K_0>0$, we must have $R(0,X_0)K_0=0$; if $K_0<0$, we must have $L(0,X_0)K_0=0$.

\noindent (ii) If $L\equiv-\infty$ and $R(t,x)=x$, then the Skorokhod problem associated with $S,L,R$ turns into the classical Skorokhod problem as in \cite{Skorokhod1}. 

\noindent (iii) If $L(t,x)=x-a$ and $R(t,x)=x$, where $a$ is a positive constant, then the Skorokhod problem associated with $S,L,R$ coincides with the Skorokhod map on $[0,a]$ studied in \cite{KLR}.

\noindent (iv) If $L(t,x)=x-r_t$ and $R(t,x)=x-l_t$, where $r,l\in D[0,\infty)$ with $l\leq r$, then the Skorokhod problem associated with $S,L,R$ corresponds to the Skorokhod problem on $[l,r]$ as in \cite{BKR}, \cite{S1}, \cite{S2}.
\end{remark}

\begin{remark}\label{remark}
Let $(X,K)=\mathbb{SP}_L^R(S)$. For any $t>0$, it is easy to check that 
\begin{align*}
K_t-K_{t-}=X_t-X_{t-}-(S_t-S_{t-}).
\end{align*}
If $K_t-K_{t-}>0$, by Definition \ref{def}, we have
\begin{align*}
K^r_t-K^r_{t-}=K_t-K_{t-}=X_t-X_{t-}-(S_t-S_{t-}) \textrm{ and } K^l_t-K^l_{t-}=0.
\end{align*}
Similarly, if $K_t-K_{t-}<0$, we have
\begin{align*}
K^l_t-K^l_{t-}=-(K_t-K_{t-})=-(X_t-X_{t-}-(S_t-S_{t-})) \textrm{ and } K^r_t-K^r_{t-}=0.
\end{align*}
Therefore,
\begin{align*}
&K^r_t-K^r_{t-}=(X_t-X_{t-}-(S_t-S_{t-}))^+,\\
&K^l_t-K^l_{t-}=(X_t-X_{t-}-(S_t-S_{t-}))^-.
\end{align*}
\end{remark}

In order to solve the Skorokhod problem with two nonlinear reflecting boundaries, we need to propose the following assumption on the functions $L,R$.
\begin{assumption}\label{ass1}
The functions $L,R:[0,\infty)\times \mathbb{R}\rightarrow \mathbb{R}$ satisfy the following conditions
\begin{itemize}
\item[(i)] For each fixed $x\in\mathbb{R}$, $L(\cdot,x),R(\cdot,x)\in D[0,\infty)$;
\item[(ii)] For any fixed $t\geq 0$, $L(t,\cdot)$, $R(t,\cdot)$ are strictly increasing;
\item[(iii)] There exists two positive constants $0<c<C<\infty$, such that for any $t\geq 0$ and $x,y\in \mathbb{R}$,
\begin{align*}
&c|x-y|\leq |L(t,x)-L(t,y)|\leq C|x-y|,\\
&c|x-y|\leq |R(t,x)-R(t,y)|\leq C|x-y|.
\end{align*} 
\item[(iv)] $\inf_{(t,x)\in[0,\infty)\times\mathbb{R}}(R(t,x)-L(t,x))>0$.
\end{itemize}
\end{assumption}

Conditions (ii) and (iii) in Assumption \ref{ass1} implies that for any $t\geq 0$, 
\begin{align*}
&\lim_{x\rightarrow -\infty} L(t,x)=-\infty, \ \lim_{x\rightarrow +\infty} L(t,x)=+\infty,\\
&\lim_{x\rightarrow -\infty} R(t,x)=-\infty, \ \lim_{x\rightarrow +\infty} R(t,x)=+\infty.
\end{align*}
Besides, suppose that $t_n\downarrow t$, $s_n\uparrow s$, $x_n\rightarrow x$. By conditions (i) and (iii) in Assumption \ref{ass1}, we have
\begin{align*}
&\lim_{n\rightarrow \infty}L(t_n,x_n)=L(t,x), \ \lim_{n\rightarrow \infty}R(t_n,x_n)=R(t,x),\\
&\lim_{n\rightarrow \infty}L(s_n,x_n)=L(s-,x), \  \lim_{n\rightarrow \infty}R(s_n,x_n)=R(s-,x).
\end{align*}

Now, given $S\in D[0,\infty)$, for any $t\geq 0$, let $\Phi_t^S$, $\Psi_t^S$ be the unique solution to  the following equations,  respectively
\begin{align}\label{phipsi}
L(t,S_t+x)=0, \ R(t,S_t+x)=0.
\end{align}
In the following, for simplicity, we always omit the superscript $S$. We first investigate the property of $\Phi$ and $\Psi$.

\begin{lemma}\label{lem}
	Under Assumption \ref{ass1}, for any given $S\in D[0,\infty)$, we have $\Phi,\Psi\in D[0,\infty)$ and 
	\begin{align*}
		\inf_{t\geq 0} (\Phi_t-\Psi_t)>0.
	\end{align*}
\end{lemma}

\begin{proof}
	We first prove that $\Phi\in D[0,\infty)$ and the proof of $\Psi\in D[0,\infty)$ is analogous. Let $\{t_n\}_{n=1}^\infty\subset [0,\infty)$ be such that $t_n\downarrow t$. We claim that $\{\Phi_{t_n}\}_{n=1}^\infty$ is bounded. Suppose that the claim does not hold. Without loss of generality, we may find a subsequence $\{\Phi_{t_{n_k}}\}_{k=1}^\infty$ such that 
	\begin{align*}
		\lim_{k\rightarrow \infty} \Phi_{t_{n_k}}=+\infty.
	\end{align*}
    Without loss of generality, we assume that $\Phi_{t_{n_k}}\geq \Phi_t$ for any $k\geq 1$. Simple calculation implies that
    \begin{align*}
    	L(t_{n_k},S_{t_{n_k}}+\Phi_{t_{n_k}})
    	=&L(t_{n_k},S_{t_{n_k}}+\Phi_{t_{n_k}})-L(t_{n_k},S_{t_{n_k}}+\Phi_{t})\\
    	&+L(t_{n_k},S_{t_{n_k}}+\Phi_{t})-L(t,S_t+\Phi_t)\\
    	\geq &c(\Phi_{t_{n_k}}-\Phi_t)+L(t_{n_k},S_{t_{n_k}}+\Phi_{t})-L(t,S_t+\Phi_t).
    \end{align*}
	It follows that 
	\begin{align*}
		\lim_{k\rightarrow \infty} L(t_{n_k},S_{t_{n_k}}+\Phi_{t_{n_k}})=+\infty,
	\end{align*}
	which is a contradiction. To show that $\Phi_{t_n}\rightarrow \Phi_t$, it suffices to prove that for any subsequence $\{t_{n'}\}\subset\{t_n\}$, one can choose a subsequence $\{t_{n''}\}\subset\{t_{n'}\}$ such that 
	$\Phi_{t_{n''}}\rightarrow \Phi_t$. Since $\{\Phi_{t_{n'}}\}$ is bounded, there exists a subsequence $\{\Phi_{t_{n''}}\}$ such that $\Phi_{t_{n''}}\rightarrow a$. Noting that 
	\begin{align*}
		L(t_{n''},S_{t_{n''}}+\Phi_{t_{n''}})=0,
	\end{align*}
	letting $n''\rightarrow \infty$, we obtain that $L(t,S_t+a)=0$, which indicates that $a=\Phi_t$. Therefore, $\Phi$ is right continuous. Similarly, we can show that $\Phi$ has left limits.
	
	It remains to prove that $\inf_{t\geq 0} (\Phi_t-\Psi_t)>0$. First, it is easy to check that $\Phi_t>\Psi_t$ for any $t\geq 0$. Set 
	\begin{align*}
		\alpha:=\inf_{(t,x)\in[0,\infty)\times\mathbb{R}}(R(t,x)-L(t,x))>0.
	\end{align*}
	It follows that 
	\begin{align*}
		-L(t,S_t+\Psi_t)=R(t,S_t+\Psi_t)-L(t,S_t+\Psi_t)>\alpha.
	\end{align*}
	Then, we obtain that 
	\begin{align*}
		\alpha<L(t,S_t+\Phi_t)-L(t,S_t+\Psi_t)\leq C(\Phi_t-\Psi_t),
	\end{align*}
	which implies that $\inf_{t\geq 0}(\Phi_t-\Psi_t)\geq \frac{\alpha}{C}$. The proof is complete. 
\end{proof}

\begin{remark}\label{continuity of phipsi}
Assume that for any $x\in\mathbb{R}$, $L(\cdot,x),R(\cdot,x)\in C[0,\infty)$ and $L,R$ satisfy (ii)-(iv) in Assumption \ref{ass1}. Given $S\in C[0,\infty)$, we can show that $\Phi,\Psi\in C[0,\infty)$.
\end{remark}

\subsection{Existence and uniqueness result}

In this subsection, we first establish uniqueness of solutions to the Skorokhod problem with two nonlinear reflecting boundaries. The proof is a relatively straightforward modification of the proof for the extended Skorokhod problem in a time-dependent interval (see, Proposition 2.8 in \cite{BKR}).

\begin{proposition}\label{uniqueness}
Suppose that $L,R$ satisfy Assumption \ref{ass1}. For any given $S\in D[0,\infty)$, there exists at most one pair of solution to the Skorokhod problem $\mathbb{SP}_L^R(S)$.
\end{proposition}

\begin{proof}
Let $(X,K)$ and $(X',K')$ be two pairs of solution to the Skorokhod problem $\mathbb{SP}_L^R(S)$. By Remark \ref{rem1}, we must have $K_0=K'_0=[-(\Phi_0)^-]\vee \Psi_0$. Consequently, $X_0=X'_0$.

Suppose that there exists some $T> 0$, such that $X_T>X'_T$. Let
\begin{align*}
\tau=\sup\{t\in[0,T]: X_t\leq X'_t\}.
\end{align*}
Since $X_0=X'_0$, $\tau$ is well-defined. We have the following two cases.

\textbf{Case 1.} $X_\tau\leq X'_\tau$. In this case, for any $t\in(\tau,T]$, we have $X_t>X'_t$. It follows that for any $t\in(\tau,T]$,
\begin{align}\label{equ2.9'}
R(t,X_t)>R(t,X'_t)\geq 0\geq L(t,X_t)>L(t,X'_t). 
\end{align}
Using condition (iii) in Definition \ref{def}, we have
\begin{align*}
K_T-K_\tau=-(K^l_T-K^l_\tau)\leq 0\leq K^{\prime,r}_T-K^{\prime,r}_\tau=K'_T-K'_\tau.
\end{align*}
Therefore, we obtain that
\begin{align*}
0<X_T-X'_T=K_T-K'_T\leq K_\tau-K'_\tau=X_\tau-X'_\tau,
\end{align*}
which contradicts the case assumption.

\textbf{Case 2.} $X_\tau>X'_\tau$. Recalling that $X_0=X'_0$, we have $\tau>0$. Besides, the definition of $\tau$ yields that 
\begin{align}\label{equ2.10}
X_{\tau-}\leq X'_{\tau-}.
\end{align}
Furthermore, in this case, \eqref{equ2.9'} holds for $t=\tau$. Similar as the proof of Case 1, we have
\begin{align*}
K_\tau-K_{\tau-}=-(K^l_\tau-K^l_{\tau-})\leq 0\leq K^{\prime,r}_\tau-K^{\prime,r}_{\tau'}=K'_\tau-K'_{\tau'}.
\end{align*}
It follows that 
\begin{align*}
0<X_\tau-X'_\tau=K_\tau-K'_\tau\leq K_{\tau-}-K'_{\tau-}=X_{\tau-}-X'_{\tau-},
\end{align*}
which contradicts \eqref{equ2.10}. 

All the above analysis indicates that $X_T\leq X'_T$ for any $T\geq 0$. Similar argument implies that $X_T\geq X'_T$ for any $T\geq 0$. Hence, $X_T=X'_T$ and consequently, $K_T=K'_T$ for any $T\geq 0$.
\end{proof}

\begin{remark}
The proof of Proposition \ref{uniqueness} is valid if $L,R$ satisfy (i) and (ii) in Assumption \ref{ass1} and $R(t,x)\geq L(t,x)$ for any $(t,x)\in[0,\infty)\times \mathbb{R}$. 
\end{remark}

Now, we state the main result in this section.

\begin{theorem}\label{main}
Suppose that Assumption \ref{ass1} holds. Given $S\in D[0,\infty)$, set 
\begin{align}\label{K}
K_t=-\max\left((-\Phi_0)^+\wedge \inf_{r\in[0,t]}(-\Psi_r),\sup_{s\in[0,t]}\left[(-\Phi_s)\wedge \inf_{r\in[s,t]}(-\Psi_r)\right] \right),
\end{align}
and $X=S+K$. Then, $(X,K)$ is the unique solution to the Skorokhod problem $\mathbb{SP}_L^R(S)$.  
\end{theorem}

\begin{remark}
	For the case as in Remark \ref{rem1} (iv), i.e., $L(t,x)=x-r_t$, $R(t,x)=x-l_t$, we  have $\Phi_t=r_t-S_t$, $\Psi_t=l_t-S_t$ and thus,
	\begin{align*}
		K_t=-\max\left((S_0-r_0)^+\wedge \inf_{r\in[0,t]}(S_r-l_r),\sup_{s\in[0,t]}\left[(S_s-r_s)\wedge \inf_{r\in[s,t]}(S_r-l_r)\right] \right),
	\end{align*}
which coincides with the results in \cite{BKR}, \cite{S1}.
\end{remark}


The proof of Theorem \ref{main} will be divided into several lemmas. We first show that the function $K$ defined by \eqref{K} is right-continuous with left limits. The proof needs the following observation as used in \cite{BKR} and \cite{S1}: $\phi\in D[0,\infty)$ if and only if the following two conditions hold for any $\varepsilon>0$:
\begin{itemize}
\item[(i)] for each $\theta_1\geq 0$, there exists some $\theta_2>\theta_1$ such that $\sup_{t_1,t_2\in[\theta_1,\theta_2)}|\phi_{t_1}-\phi_{t_2}|\leq \varepsilon$; 
\item[(i)] for each $\theta_2> 0$, there exists some $0\leq \theta_1<\theta_1$ such that $\sup_{t_1,t_2\in[\theta_1,\theta_2)}|\phi_{t_1}-\phi_{t_2}|\leq \varepsilon$.
\end{itemize}

\begin{lemma}\label{RCLL}
If $\Phi,\Psi\in D[0,\infty)$, then $K$ defined by \eqref{K} belongs to $D[0,\infty)$.
\end{lemma}

\begin{proof}
For any $0\leq s\leq t$, set 
\begin{align*}
&H^{\Phi,\Psi}(t):=(-\Phi_0)^+\wedge \inf_{r\in[0,t]}(-\Psi_r)\\
&R^{\Phi,\Psi}_t(s):=(-\Phi_s)\wedge \inf_{r\in[s,t]}(-\Psi_r),\\
&C^{\Phi,\Psi}(t):=\sup_{s\in[0,t]}R^{\Phi,\Psi}_t(s).
\end{align*}
For simplicity, we always omit the superscript $\Phi,\Psi$. It is easy to check that
\begin{align}\label{equa2.2}
(-\Phi_t)\wedge(-\Psi_t)\leq C(t)\leq -\Psi_t.
\end{align}
For any $0\leq \theta_1<\theta_2$, let $t_1,t_2$ be in $[\theta_1,\theta_2)$ with $t_1\leq t_2$ and 
\begin{align*}
a:=\sup_{s,u\in[\theta_1,\theta_2)}|\Phi_s-\Phi_u|+\sup_{s,u\in[\theta_1,\theta_2)}|\Psi_s-\Psi_u|.
\end{align*}
Then, for any $s\in[0,t_1]$, we have $R_{t_2}(s)\leq R_{t_1}(s)$ and
\begin{align*}
\sup_{s\in(t_1,t_2]}R_{t_2}(s)\leq \sup_{s\in(t_1,t_2]}(-\Phi_s)\wedge(-\Psi_s)\leq (-\Phi_{t_1})\wedge(-\Psi_{t_1})+a.
\end{align*}
It follows that 
\begin{align*}
C(t_2)=&\sup_{s\in[0,t_1]}R_{t_2}(s)\vee\sup_{s\in(t_1,t_2]}R_{t_2}(s)\\
\leq &\sup_{s\in[0,t_1]}R_{t_1}(s)\vee[(-\Phi_{t_1})\wedge(-\Psi_{t_1})+a]\\
\leq &\sup_{s\in[0,t_1]}R_{t_1}(s)+a=C(t_1)+a.
\end{align*}
On the other hand, noting \eqref{equa2.2}, we obtain that 
\begin{align*}
C(t_1)-a\leq &C(t_1)\wedge(-\Psi_{t_1}-a)\\
\leq &\sup_{s\in[0,t_1]}[R_{t_1}(s)\wedge \inf_{r\in(t_1,t_2]}(-\Psi_r)]\\
=&\sup_{s\in[0,t_1]}R_{t_2}(s)\leq \sup_{s\in[0,t_2]}R_{t_2}(s)=C(t_2).
\end{align*}
All the above analysis indicates that $|C(t_1)-C(t_2)|\leq a$. Besides, simple calculation yields that
\begin{align*}
H(t_2)\leq & H(t_1)=H(t_1)\wedge(-\Psi_{t_1}-a+a)\\
\leq &H(t_1)\wedge (-\Psi_{t_1}-a)+a\\
\leq &H(t_1)\wedge\inf_{s\in(t_1,t_2]}(-\Psi_s)+a=H(t_2)+a,
\end{align*}
which indicates that $|H(t_1)-H(t_2)|\leq a$. It is easy to check that the following inequality holds for any $x_i,y_i\in\mathbb{R}$, $i=1,2$,
\begin{align*}
|x_1\vee x_2-y_1\vee y_2|\leq |x_1-x_2|\vee|y_1-y_2|.
\end{align*}
Then, we obtain that 
\begin{align*}
|K_{t_1}-K_{t_2}|\leq |H(t_1)-H(t_2)|\vee |C(t_1)-C(t_2)|\leq a.
\end{align*}
Since $t_1\leq t_2$ are arbitrarily chosen in $[\theta_1,\theta_2)$, we deduce that 
\begin{align*}
\sup_{t_1,t_2\in[\theta_1,\theta_2)}|K_{t_1}-K_{t_2}|\leq \sup_{s,u\in[\theta_1,\theta_2)}|\Phi_s-\Phi_u|+\sup_{s,u\in[\theta_1,\theta_2)}|\Psi_s-\Psi_u|.
\end{align*}
Consequently, $K\in D[0,\infty)$.
\end{proof}

\begin{remark}
The proof of Lemma \ref{RCLL} also shows that the oscillation of $K$ can be dominated by the oscillation of $\Phi$ and $\Psi$ on closed interval $[\theta_1,\theta_2]$, i.e.,
\begin{align*}
\sup_{t_1,t_2\in[\theta_1,\theta_2]}|K_{t_1}-K_{t_2}|\leq \sup_{s,u\in[\theta_1,\theta_2]}|\Phi_s-\Phi_u|+\sup_{s,u\in[\theta_1,\theta_2]}|\Psi_s-\Psi_u|.
\end{align*}
Therefore, if $\Phi,\Psi\in C[0,\infty)$, we have $K\in C[0,\infty)$. Under the assumption as in Remark \ref{continuity of phipsi}, for any given $S\in C[0,\infty)$, each component of solution $(X,K)$ to the Skorokhod problem $\mathbb{SP}_L^R(S)$ is continuous.
\end{remark}

Now, we define the following pair of times
\begin{align}\label{eq2.1}
\sigma^*=\inf\{t>0|\Phi_t\leq 0\}, \ \tau^*=\inf\{t>0|\Psi_t\geq 0\}.
\end{align}

\begin{remark}\label{rem2.2}
(i) Noting that $a:=\inf_t(\Phi_t-\Psi_t)>0$ as shown in Lemma \ref{lem}, there are three cases
\begin{align}\label{eq2.6}
\textrm{either } \sigma^*=\tau^*=\infty, \ \sigma^*<\tau^* \textrm{ or } \sigma^*>\tau^*.
\end{align}
Especially, for the case $\sigma^*=\tau^*=\infty$, we have $\Psi_t<0<\Phi_t$ for any $t\geq 0$. Consequently, $K_t= 0$ and $L(t,S_t)< 0< R(t,S_t)$ for any $t$. Hence, $(S,0)$ is the solution to the  Skorokhod problem $\mathbb{SP}_L^R(S)$. In the following of this section, we only consider the other two cases.

\noindent (ii) For any $t\in[0,\sigma^*\wedge\tau^*)$, we have $\Psi_t<0<\Phi_t$. It follows that $K_t=0$ when $t\in[0,\sigma^*\wedge\tau^*)$.
\end{remark}

For the case $\tau^*>\sigma^*$, we set $\tau_0=0$, $\sigma_0=\sigma^*$ and for $k\geq 1$, we set
\begin{equation}\label{eq2.8}
\tau_k=\min\{ t>\sigma_{k-1}| \inf_{s\in[\sigma_{k-1},t]}\Phi_s\leq \Psi_s\}
\end{equation}
and 
\begin{equation}\label{eq2.9'}
\sigma_k=\min\{ t>\tau_{k}| \sup_{s\in[\tau_k,t]}\Psi_s\geq \Phi_s\}.
\end{equation}
Since $\Phi,\Psi$ are right continuous, $\tau_k$, $\sigma_k$ are well-defined. For the case $\tau^*<\sigma^*$, we set $\tau_0=\tau^*$ and define $\sigma_k$ by \eqref{eq2.9'} for all $k\geq 0$, and we define $\tau_k$ by \eqref{eq2.8} for all $k\geq 1$.

It is easy to check that  for both cases $\tau^*<\sigma^*$ and $\tau^*>\sigma^*$,  the following two inequalities hold for $k\geq 1$,
\begin{equation}\label{eq2.10}
\inf_{s\in[\sigma_{k-1},t]}\Phi_s>\Psi_t \textrm{ for any } t\in[\sigma_{k-1},\tau_k),
\end{equation}
\begin{equation}\label{eq2.11}
\inf_{s\in[\sigma_{k-1},\tau_k]}\Phi_s\leq \Psi_{\tau_k},
\end{equation}
Furthermore, the following two inequalities hold for any $k\geq 1$ if $\tau^*>\sigma^*$ and hold for any $k\geq0$ if $\tau^*<\sigma^*$,
\begin{equation}\label{eq2.12}
\Phi_t>\sup_{s\in[\tau_k,t]}\Psi_s \textrm{ for any } t\in[\tau_k,\sigma_k),
\end{equation}
\begin{equation}\label{eq2.13}
\Phi_{\sigma_k}\leq \sup_{s\in[\tau_k,\sigma_k]}\Psi_s.
\end{equation}
Finally, for the case $\tau^*>\sigma^*$, we have $\Psi_t\leq 0$ for any $t\in[0,\sigma_0]$ and 
\begin{equation}\label{eq2.9}
\Phi_{\sigma_0}\leq 0.
\end{equation}
For the case $\tau^*<\sigma^*$, we have $\Phi_t\geq 0$ for any $t\in[0,\tau_0]$ and 
\begin{equation}\label{eq2.9''}
\Psi_{\tau_0}\geq 0.
\end{equation}
By \eqref{eq2.10}, it is easy to check that $\Phi_s>\Psi_t$ for any $\sigma_{k-1}\leq s\leq t<\tau_k$, $k\geq 1$. It follows that 
\begin{equation}\label{eq2.14}
\Phi_s>\sup_{t\in[s,\tau_k)}\Psi_t \textrm{ for any } s\in[\sigma_{k-1},\tau_k).
\end{equation}

It is clear that \begin{align*}
	0\leq \tau_0\leq \sigma_0<\tau_1<\sigma_1<\tau_2<\sigma_2<\cdots.
\end{align*}
We first show that $\tau_{k}$, $\sigma_{k}$ approach infinity as $k$ goes to infinity.

\begin{proposition}\label{prop2.1}
	Under Assumption \ref{ass1}, we have
	\begin{align*}
		\lim_{k\rightarrow \infty}\tau_k=\lim_{k\rightarrow \infty}\sigma_k=\infty.
	\end{align*}
\end{proposition}

\begin{proof}
	The proof is similar with the one for Proposition 3.1 in \cite{KLR}. For readers' convenience, we give a short proof here. By Lemma \ref{lem}, $a:=\inf_{t\geq 0}(\Phi_t-\Psi_t)>0$. Suppose that 
	\begin{align*}
		\lim_{k\rightarrow \infty}\tau_k=\lim_{k\rightarrow \infty}\sigma_k=t^*<\infty.
	\end{align*}
	For any $k\geq 1$, there exists some $\rho_k\in[\sigma_{k-1},\tau_k]$, such that 
	\begin{align*}
		\inf_{t\in[\sigma_{k-1},\tau_k]}\Phi_t\geq \Phi_{\rho_k}-\frac{a}{2}.
	\end{align*}
	Recalling \eqref{eq2.11} and the definition of $a$, we have
	\begin{align*}
		\Psi_{\tau_k}\geq \Psi_{\rho_k}+\frac{a}{2}.
	\end{align*}
	Letting $k\rightarrow \infty$, it follows that $\Psi$ does not have left limit at $t^*$, which is a contradiction. Therefore, we have
	\begin{align*}
		\lim_{k\rightarrow \infty}\tau_k=\lim_{k\rightarrow \infty}\sigma_k=\infty.
	\end{align*}
\end{proof}

By (iii) in Definition \ref{def}, the second component of solution to the Skorokhod problem with two nonlinear reflecting boundaries is of bounded variation. In the following two propositions, we show that $K$ is piecewise monotone. Therefore, $K$ defined by \eqref{K} is a bounded variation function.

\begin{proposition}\label{prop2.2}
Under Assumption \ref{ass1}, for any $k\geq 1$ and $t\in[\sigma_{k-1},\tau_k)$, we have
\begin{align*}
-K_t=\sup_{s\in[\sigma_{k-1},t]}(-\Phi_s).
\end{align*}
\end{proposition}

\begin{proof}
The proof is similar with the one for Lemma 2.7 in \cite{S2}. For readers' convenience, we give a short proof here. For any $k\geq 1$ and $t\in[\sigma_{k-1},\tau_k)$, set 
\begin{align*}
I^1_t&=(-\Phi_0)^+\wedge \inf_{r\in[0,t]}(-\Psi_r),\\
I^{2,k-1}_t&=\sup_{s\in[0,\tau_{k-1}]}\left[(-\Phi_s)\wedge \inf_{r\in[s,t]}(-\Psi_r)\right],\\
I^{3,k-1}_t&=\sup_{s\in[\tau_{k-1},\sigma_{k-1})}\left[(-\Phi_s)\wedge \inf_{r\in[s,t]}(-\Psi_r)\right],\\
I^{4,k-1}_t&=\sup_{s\in[\sigma_{k-1},t]}\left[(-\Phi_s)\wedge \inf_{r\in[s,t]}(-\Psi_r)\right].
\end{align*}
It is obvious that $-K_t=I^1_t\vee I^{2,k-1}_t\vee I^{3,k-1}_t\vee I^{4,k-1}_t$. 

\textbf{Case 1.} $k\geq 2$ if $\tau^*>\sigma^*$ and $k\geq 1$ if $\tau^*<\sigma^*$. By \eqref{eq2.13}, it follows that 
\begin{align}\label{eq2.20}
\inf_{s\in[\tau_{k-1},\sigma_{k-1}]}(-\Psi_s)\leq -\Phi_{\sigma_{k-1}}\leq \sup_{s\in[\sigma_{k-1},t]}(-\Phi_s).
\end{align}
Therefore, we have
\begin{align*}
I^1_t&\leq \inf_{r\in[0,t]}(-\Psi_r)\leq \inf_{s\in[\tau_{k-1},\sigma_{k-1}]}(-\Psi_s)\leq \sup_{s\in[\sigma_{k-1},t]}(-\Phi_s),\\
I^{2,k-1}_t&\leq \sup_{s\in[0,\tau_{k-1}]} \inf_{r\in[s,t]}(-\Psi_r)\leq \inf_{s\in[\tau_{k-1},\sigma_{k-1}]}(-\Psi_s)\leq \sup_{s\in[\sigma_{k-1},t]}(-\Phi_s).
\end{align*}
Recalling \eqref{eq2.12}, for any $t\in[\tau_{k-1},\sigma_{k-1})$, we have $-\Phi_t<\inf_{r\in[\tau_{k-1},t]}(-\Psi_r)$. Then we obtain that for any $t\in[\sigma_{k-1},\tau_k)$,
\begin{align*}
I^{3,k-1}_t&\leq \sup_{s\in[\tau_{k-1},\sigma_{k-1})}\left[\inf_{r\in[\tau_{k-1},t]}(-\Psi_r)\wedge \inf_{r\in[s,t]}(-\Psi_r)\right]\\
&\leq \inf_{r\in[\tau_{k-1},\sigma_{k-1}]}(-\Psi_r)\leq \sup_{s\in[\sigma_{k-1},t]}(-\Phi_s),
\end{align*}
where we have used \eqref{eq2.20} in the last inequality. It follows from \eqref{eq2.14} that for any $k\geq 1$
\begin{align}\label{eq2.14'}
I^{4,k-1}_t=\sup_{s\in[\sigma_{k-1},t]}(-\Phi_s).
\end{align}
Then, in this case, $-K_t=\sup_{s\in[\sigma_{k-1},t]}(-\Phi_s)$.

\textbf{Case 2.} $k=1$ if $\tau^*>\sigma^*$. In this case, for any $t\in[0,\sigma_0)$,  we have $\Psi_t\leq 0\leq \Phi_t$, $\Psi_{\sigma_0}\leq 0$, $\Phi_{\sigma_0}\leq 0$. Therefore, it is easy to check that for any $t\in[\sigma_0,\tau_1)$,
\begin{align*}
&I^1_t\leq (-\Phi_0)^+=0, \ I^{2,0}_t=(-\Phi_0)\wedge (-\Psi_0)\leq 0,\\
&I^{3,0}_t=\sup_{s\in[0,\sigma_0)}[(-\Phi_s)\wedge \inf_{r\in[s,t]}(-\Psi_r)]\leq 
\sup_{s\in[0,\sigma_0)}[(-\Phi_s)\wedge \inf_{r\in[s,\sigma_0]}(-\Psi_r)]=\sup_{s\in[0,\sigma_0)}(-\Phi_s)\leq 0.
\end{align*}
Recalling \eqref{eq2.14'}, we have $I^{4,0}_t=\sup_{s\in[\sigma_0,t]}(-\Phi_s)\geq -\Phi_{\sigma_0}\geq 0$. Therefore, in this case, $-K_t=\sup_{s\in[\sigma_0,t]}(-\Phi_s)$. The proof is complete.
\end{proof}

\begin{proposition}\label{prop2.3}
Suppose Assumption \ref{ass1} holds. If $k\geq 1$ or $\tau^*<\sigma^*$ and $k=0$, then for any  $t\in[\tau_k,\sigma_k)$, we have
\begin{align*}
-K_t=\inf_{s\in[\tau_{k},t]}(-\Psi_s).
\end{align*}
\end{proposition}

\begin{proof}
The proof is similar with the one for Lemma 2.8 in \cite{S2}. For readers' convenience, we give a short proof here. Let $t\in[\tau_k,\sigma_k)$. Set 
\begin{align*}
I^{5,k}_t=\sup_{s\in[\tau_k,t]}\left[(-\Phi_s)\wedge \inf_{r\in[s,t]}(-\Psi_r)\right].
\end{align*}
Then, we have $-K_t=I^1_t\vee I^{2,k}_t\vee I^{5,k}_t$, where $I^1$, $I^{2,k}$ are defined in the proof of Proposition \ref{prop2.2}. It is easy to check that 
\begin{align*}
I^1_t&\leq \inf_{s\in[0,t]}(-\Psi_s)\leq \inf_{s\in[\tau_k,t]}(-\Psi_s),\\
I^{2,k}_t&\leq \sup_{s\in[0,\tau_{k}]} \inf_{r\in[s,t]}(-\Psi_r)\leq \inf_{s\in[\tau_k,t]}(-\Psi_s).
\end{align*}
By \eqref{eq2.12}, we have
\begin{align*}
I^{5,k}_t\leq \sup_{s\in[\tau_k,t]}\left[\inf_{r\in[\tau_k,s]}(-\Psi_r)\wedge \inf_{r\in[s,t]}(-\Psi_r)\right]\leq \inf_{s\in[\tau_k,t]}(-\Psi_s).
\end{align*}
All the above analysis indicates that $-K_t\leq \inf_{s\in[\tau_k,t]}(-\Psi_s)$. Now, we are in a position to prove the inverse inequality. It is sufficient to prove that for $k\geq 1$,
\begin{align}\label{eq2.36}
I^{2,k}_t\geq  \inf_{s\in[\tau_k,t]}(-\Psi_s), \textrm{ for } t\in[\tau_k,\sigma_k)
\end{align}
and  for $\tau^*<\sigma^*$, $k=0$
\begin{align}\label{eq2.40}
I^1_t=\inf_{r\in[\tau_0,t]}(-\Psi_r), \textrm{ for } t\in[\tau_0,\sigma_0).
\end{align}
We first prove \eqref{eq2.36}. For any fixed $\varepsilon>0$ and $k\geq 1$, there exists some $\rho\in[\sigma_{k-1},\tau_k]$, such that 
\begin{align*}
\inf_{s\in[\sigma_{k-1},\tau_k]}\Phi_s\geq \Phi_\rho-\varepsilon.
\end{align*}
Together with \eqref{eq2.11}, we have
\begin{align*}
-\Phi_\rho\geq -\Psi_{\tau_k}-\varepsilon.
\end{align*}
Recalling \eqref{eq2.14}, we obtain that 
\begin{align*}
I^{2,k}_t\geq &(-\Phi_\rho)\wedge \inf_{r\in[\rho,t]}(-\Psi_r)\\
\geq &(-\Phi_\rho)\wedge \inf_{r\in[\rho,\tau_k)}(-\Psi_r)\wedge \inf_{r\in[\tau_k,t]}(-\Psi_r)\\
= &(-\Phi_\rho)\wedge \inf_{r\in[\tau_k,t]}(-\Psi_r)\\
\geq &(-\Psi_{\tau_k}-\varepsilon)\wedge \inf_{r\in[\tau_k,t]}(-\Psi_r)\geq  \inf_{r\in[\tau_k,t]}(-\Psi_r)-\varepsilon.
\end{align*}
Since $\varepsilon$ can be arbitrarily small, \eqref{eq2.36} holds true for any $k\geq 1$. It remains to prove that \eqref{eq2.40} holds when $\tau^*<\sigma^*$ and $k=0$. Indeed, since $\sigma^*>0$, $\Phi_0>0$. Besides, as $\tau_0=\tau^*$, we have $\sup_{r\in[0,\tau_0)}\Psi_r\leq 0$ and $\sup_{r\in[\tau_0,t]}\Psi_r\geq \Psi_{\tau_0}\geq 0$. Then, we may check that 
\begin{align*}
I^1_t=0\wedge \inf_{r\in[0,\tau_0)}(-\Psi_r)\wedge \inf_{r\in[\tau_0,t]}(-\Psi_r)=\inf_{r\in[\tau_0,t]}(-\Psi_r).
\end{align*}
The proof is complete.
\end{proof}

Combining Remark \ref{rem2.2}, Proposition \ref{prop2.2} and Proposition \ref{prop2.3}, we have the following representation for $K$, which is a generalization of Theorem 2.2 in \cite{S1} and Theorem 2.6 in \cite{S2}.
\begin{theorem}\label{thm2.2}
Under Assumption \ref{ass1}, Let $K$ be defined by \eqref{K}. If $\tau^*>\sigma^*$, then
\begin{equation}\label{eq2.21}
-K_t=\begin{cases}
0, &t\in[0,\sigma_0),\\
\sup_{s\in[\sigma_{k-1},t]}(-\Phi_s), &t\in[\sigma_{k-1},\tau_k), \ k\geq 1,\\
\inf_{s\in[\tau_{k},t]}(-\Psi_s), &t\in[\tau_k,\sigma_k), \ k\geq 1.
\end{cases}
\end{equation}
If $\tau^*<\sigma^*$, then
\begin{equation}\label{eq2.22}
-K_t=\begin{cases}
0, &t\in[0,\tau_0),\\
\inf_{s\in[\tau_{k},t]}(-\Psi_s), &t\in[\tau_k,\sigma_k), \ k\geq 0,\\
\sup_{s\in[\sigma_{k-1},t]}(-\Phi_s), &t\in[\sigma_{k-1},\tau_k), \ k\geq 1.
\end{cases}
\end{equation}
\end{theorem}

\begin{remark}\label{rem}
It is easy to check that, for any $k\geq 0$, $K_{\sigma_{k}}=\Phi_{\sigma_k}$ and if $k\geq 1$ or $\tau^*<\sigma^*$ and $k=0$, $K_{\tau_k}=\Psi_{\tau_k}$.
\end{remark}

Now, we show that $K$ can be represented as the difference between two nondecreasing functions which only increase when $R(\cdot,X_\cdot)$, $L(\cdot,X_\cdot)$ hit $0$. 
\begin{theorem}\label{thm2.3}
Suppose that Assumption \ref{ass1} holds. Let $X=S+K$. Then, we have
\begin{itemize}
\item[(1)] $K\in BV[0,\infty)$,
\item[(2)] $L(t,X_t)\leq 0\leq R(t,X_t)$ for any $t\geq 0$,
\item[(3)] $|K|_t=\int_0^t I_{\{L(s,X_s)= 0\textrm{ or } R(s,X_s)=0\}}d|K|_s$,
\item[(4)] $K_t=\int_0^t I_{\{R(s,X_s)=0\}}d|K|_s-\int_0^t I_{\{L(s,X_s)= 0\}}d|K|_s$.
\end{itemize}
\end{theorem}

\begin{proof}
(1) is a direct consequence of Theorem \ref{thm2.2}.  \eqref{K} can be written as
\begin{align*}
K_t=\min\left([-(\Phi_0)^-]\vee \sup_{r\in[0,t]}\Psi_r,\inf_{s\in[0,t]}\left[\Phi_s\vee \sup_{r\in[s,t]}\Psi_r\right] \right).
\end{align*}
Recalling that $\Phi\geq \Psi$, it follows that 
\begin{align*}
\inf_{s\in[0,t]}\left[\Phi_s\vee \sup_{r\in[s,t]}\Psi_r\right]\leq \Phi_t\vee\Psi_t=\Phi_t,
\end{align*}
which implies that $K_t\leq \Phi_t$. On the other hand, it is easy to check that 
\begin{align*}
&[-(\Phi_0)^-]\vee \sup_{r\in[0,t]}\Psi_r\geq \Psi_t,\\
&\Phi_s\vee \sup_{r\in[s,t]}\Psi_r\geq \Psi_t, \textrm{ for any } s\in[0,t].
\end{align*}
Consequently, we have $K_t\geq \Psi_t$. Therefore, by the definition for $\Phi,\Psi$, we have
\begin{align*}
&L(t,X_t)=L(t,S_t+K_t)\leq L(t,S_t+\Phi_t)=0,\\
&R(t,X_t)=R(t,S_t+K_t)\geq R(t,S_t+\Psi_t)=0.
\end{align*}

Motivated by the proof of Theorem 3.4 in \cite{KLR}, we only prove (3), (4) for the case $\tau^*>\sigma^*$ since the case $\tau^*<\sigma^*$ can be proved similarly. Since $K_t=0$ when $t\in[0,\sigma_0)$, we only focus on the case that $t\geq \sigma_0$. We claim that for $t\geq \sigma_0$,
\begin{equation}\label{eq3.33}
R(t,X_t)=0\textrm{ implies that } t\in[\tau_k,\sigma_k) \textrm{ for some } k\geq 1,
\end{equation}
and 
\begin{equation}\label{eq3.34}
L(t,X_t)=0\textrm{ implies that } t\in[\sigma_{k-1},\tau_k) \textrm{ for some } k\geq 1.
\end{equation}
 We first prove \eqref{eq3.33}. Suppose that $t\in[\sigma_{k-1},\tau_k)$ for some $k\geq 1$. Recalling \eqref{eq2.10} and \eqref{eq2.21}, we have
\begin{align*}
-\Psi_t>\sup_{s\in[\sigma_{k-1},t]}(-\Phi_s)=-K_t.
\end{align*} 
It follows that 
\begin{align*}
R(t,X_t)>R(t,S_t+\Psi_t)=0,
\end{align*}
which implies that \eqref{eq3.33} holds. Now, suppose that $t\in[\tau_k,\sigma_k)$ for some $k\geq 1$. By \eqref{eq2.12} and \eqref{eq2.21}, we obtain that 
\begin{align*}
-\Phi_t<\inf_{s\in[\tau_k,t]}(-\Psi_s)=-K_t.
\end{align*}
Consequently, we have
\begin{align*}
L(t,X_t)<L(t,S_t+\Phi_t)=0.
\end{align*}
Therefore, \eqref{eq3.34} holds. Then, if assertion (3) is true, by \eqref{eq2.21}, \eqref{eq3.33} and \eqref{eq3.34}, we derive that (4) is true.

Now, it remains to prove (3). Set 
\begin{align*}
A:=\{t\geq \sigma_0:L(t,X_t)<0<R(t,X_t)\}.
\end{align*}
It suffices to prove that $\int_Ad|K|_t=0$. For $t\in A$, we define
\begin{align*}
&l_t:=L(t,X_t), \ r_t:=R(t,X_t),\\
&\alpha_t:=\inf\{s\in[\sigma_0,t]:(s,t]\subset A\},\\
&\beta_t:=\sup\{s\in[t,\infty):[t,s)\subset A\}.
\end{align*}
By the right-continuity of $l$ and $r$, we have $\beta_t\notin A$ while $\alpha_t$ might or might not belong to $A$. Besides, we have $\alpha_t\leq t<\beta_t$, which implies that $(\alpha_t,\beta_t)$ is nonempty. The above analysis implies that $A$ has the following representation
\begin{align*}
A=\left(\cup_{t\in I}(\alpha_t,\beta_t)\right)\cup\{\alpha_t:t\in J\},
\end{align*}
where $I$ is a countable subset of $[0,\infty)$ and $J\subset I$.

We first show that $\int_{(\alpha_t,\beta_t)}d|K|_s=0$ for any $t\in I$. Note that for any $s\in(\alpha_t,\beta_t)$, we have $l_s<0<r_s$. Recalling the definition of $\Phi$, $\Psi$ and Remark \ref{rem}, for any $k\geq 1$, we have
\begin{align}\label{eq3.31}
r_{\tau_k}=R(\tau_k,S_{\tau_k}+\Psi_{\tau_k})=0, \ l_{\sigma_{k-1}}=L(\sigma_{k-1},S_{\sigma_{k-1}}+\Phi_{\sigma_{k-1}})=0.
\end{align}
Therefore, there are only two possibilities: either $(\alpha_t,\beta_t)\subset (\tau_k,\sigma_k)$ or $(\alpha_t,\beta_t)\subset(\sigma_{k-1},\tau_k)$ for some $k\geq 1$. We only consider the second case as the first case is analogous. It is enough to show that $K$ is a constant on $[a_t,b_t]$ for any $[a_t,b_t]\subset(\alpha_t,\beta_t)$. Recalling that when $t\in[\sigma_{k-1},\tau_k)$, $K_t=\inf_{s\in[\sigma_{k-1},t]}\Phi_s$. Set 
\begin{align*}
\rho=\inf\{s\in[a_t,b_t]: K_s<K_{a_t}\}.
\end{align*}
Suppose that $\rho<\infty$. The right-continuity of $K$ yields that $K_s=K_{a_t}$ for any $s\in[a_t,\rho)$ and either $K_\rho=\Phi_\rho<K_{a_t}$ or $K_\rho=\Phi_\rho=K_{a_t}$. In either case, 
\begin{align*}
l_\rho=L(\rho,X_\rho)=L(\rho,S_\rho+\Phi_\rho)=0,
\end{align*}
which contradicts the fact that $\rho\in A$. Hence, $\rho=\infty$ and $K$ is a constant on $[a_t,b_t]$.

To complete the proof, it suffices to show that for any $\alpha_t\in A$ with $t\in J$, $K$ is continuous at $a_t$. Recalling \eqref{eq3.31}, there exists some $k\geq 1$, such that either $\alpha_t\subset (\tau_k,\sigma_k)$ or $\alpha_t\subset(\sigma_{k-1},\tau_k)$. By the definition of $\alpha_t$, we may find a sequence $\{\gamma_n\}_{n=1}^\infty \subset (0,\alpha_t)\cap A^c$ such that $\gamma_n\uparrow \alpha_t$.

We first consider the case that $l_{\gamma_n}=0$ or equivalently,  $K_{\gamma_n}=\Phi_{\gamma_n}$ for infinitely many values of $n$. Applying \eqref{eq3.34}, we have $\gamma_n\in [\sigma_{k-1},\tau_k)$ for some $k\geq 1$. There exists some $k^*$ independent of $n$, such that for $n$ large enough, $\gamma_n\in[\sigma_{k^*-1},\tau_{k^*})$ and $\alpha_t\in (\sigma_{k^*-1},\tau_{k^*})$. 
Therefore, we have
\begin{align*}
\Phi_{\gamma_n}=K_{\gamma_n}=\inf_{s\in[\sigma_{k^*-1},\gamma_n]}\Phi_s.
\end{align*}
Letting $n\rightarrow \infty$ implies that
\begin{align*}
\Phi_{\alpha_t-}=K_{\alpha_t-}=\inf_{s\in[\sigma_{k^*-1},\alpha_t)}\Phi_s.
\end{align*}
Since $\alpha_t\in [\sigma_{k^*-1},\tau_{k^*})$, we have $K_{\alpha_t}=\inf_{s\in[\sigma_{k^*-1},\alpha_t]}\Phi_s$, which yields that $K_{\alpha_t}\leq K_{\alpha_t-}$. Suppose that $K_{\alpha_t}<K_{\alpha_t-}$, which implies that $K_{\alpha_t}=\Phi_{\alpha_t}$. This leads to the following equality
\begin{align*}
l_{\alpha_t}=L(\alpha_t,X_{\alpha_t})=L(\alpha_t,S_{\alpha_t}+\Phi_{\alpha_t})=0,
\end{align*} 
which contradicts the fact that $\alpha_t\in A$. Therefore, we have $K_{\alpha_t}= K_{\alpha_t-}$, that is, $K$ is continuous at $\alpha_t$.

For the case that $l_{\gamma_n}=0$ does not hold for infinitely many values of $n$, then $r_{\gamma_n}=0$ must hold for infinitely many values of $n$. By a similar analysis as above, we can also show that $K$ is continuous at $\alpha_t$. The proof is complete.
\end{proof}

\begin{proof}[Proof of Theorem \ref{main}]
The uniqueness of solution to nonlinear Skorokhod problem is a direct consequence of Proposition \ref{uniqueness}. Let $K$ be defined as \eqref{K} and set $X_t=S_t+K_t$,
\begin{align*}
K^r_t=\int_0^t I_{\{R(s,X_s)=0\}}d|K|_s, \ K^l_t=\int_0^t I_{\{L(s,X_s)=0\}}d|K|_s.
\end{align*}
Clearly, $K^r$, $K^l$ are nondecreasing functions. By Theorem \ref{thm2.3}, we have $L(t,X_t)\leq 0\leq R(t,X_t)$, $K_t=K^r_t-K^l_t$ for any $t\geq 0$, and 
\begin{align*}
\int_0^\infty I_{\{L(s,X_s)<0\}}dK^l_s=0, \  \int_0^\infty I_{\{R(s,X_s)>0\}}dK^r_s=0.
\end{align*}
That is, $(X,K)$ is the solution to the Skorokhod problem $\mathbb{SP}_L^R(S)$.
\end{proof}

\begin{remark}\label{single}
Suppose that $L\equiv-\infty$. Then, the Skorokhod problem with two nonlinear reflecting boundaries turns into the Skorokhod problem with one constraint. More precisely, given $S\in D[0,\infty)$, we need to find pair of functions $(X,K)\in D[0,\infty)\times I[0,\infty)$ such that
\begin{itemize}
\item[(i)] $X_t=S_t+K_t$;
\item[(ii)] $R(t,X_t)\geq 0$;
\item[(iii)] $K_{0-}=0$ and $K$ is a nondecreasing function satisfying
\begin{align*}
 \int_0^\infty I_{\{R(s,X_s)>0\}}dK_s=0.
\end{align*}
\end{itemize}
For simplicity, such pair of functions is written as $(X,K)=\mathbb{SP}^R(S)$.

Since $L\equiv -\infty$, $\Phi$ maybe interpreted as $+\infty$. Recalling \eqref{K}, we have 
$$
K_t=\sup_{s\in[0,t]}\Psi_s^+.
$$
Especially, if $R(t,x)=x$, the Skorokhod problem with nonlinear constraint degenerates to the classical Skorokhod problem. In this case, we have $\Psi_t=-S_t$. Consequently, $K_t=\sup_{s\in[0,T]}S_t^-$, which coincides with the result in \cite{Skorokhod1}.
\end{remark}

\begin{remark}\label{rem4.9}
Suppose that $(X,K)$ solves the Skorokhod problem $\mathbb{SP}_L^R(S)$ and $K$ admits the decomposition $K=K^r-K^l$. Then, $(X,K^r)$ may be interpreted as the solution to the  Skorokhod problem with nonlinear constraint $\mathbb{SP}^R(S-K^l)$. For any $t\geq 0$, let $\Psi^r_t$ be the solution to the following equation
$$
R(t,S_t-K^l_t+\Psi^r_t)=0.
$$
Then, we have $K^r_t=\sup_{s\in[0,t]}(\Psi^r_s)^+$. On the other hand, $(-X,K^l)$ may be regarded as the solution to the Skorokhod problem with nonlinear constraint $\mathbb{SP}^{\tilde{L}}(-S-K^r)$, where $\tilde{L}(t,x):=-L(t,-x)$. For any $t\geq 0$, let $\Phi^l_t$ be the solution to the following equation
$$
\tilde{L}(t,-S_t-K^{r}_t+\Phi^l_t)=-L(t,S_t+K^l_t-\Phi^l_t)=0.
$$
Then, we have $K^l_t=\sup_{s\in[0,t]}(\Phi^l_s)^+$.
\end{remark}

\begin{remark}
	It is worth pointing out that (iv) in Assumption \ref{ass1} is necessary for the construction of $K$ to be of bounded variation. The readers may refer to Example 2.1 in \cite{S1} for a counterexample.
\end{remark}

\section{Property of solutions to Skorokhod problems with two nonlinear reflecting boundaries}

\subsection{Non-anticipatory properties}
In this subsection, suppose $(X,K)$ is the solution to a Skorokhod problem with two nonlinear reflecting boundaries. We investigate if the  pair of shift functions is still the solution to some other Skorokhod problem. For this purpose, for any fixed $d\geq 0$, we define two operators $T_d,H_d:D[0,\infty)\rightarrow D[0,\infty)$ as follows:
\begin{align}\label{tdhd}
(T_d(\psi))_t:=\psi_{d+t}-\psi_d,\ (H_d(\psi))_t:=\psi_{d+t}, \ t\geq 0.
\end{align}
Moreover, we define two functions $L^d,R^d:[0,\infty)\times\mathbb{R}\rightarrow \mathbb{R}$ as follows:
\begin{align*}
L^d(t,x):=L(t+d,x),\ R^d(t,x):=R(t+d,x).
\end{align*} 

\begin{theorem}\label{thm3.1}
Under Assumption \ref{ass1}, for a given $S\in D[0,\infty)$, if $(X,K)$ solves the Skorokhod problem $\mathbb{SP}_{L}^R(S)$, then $(H_d(X),T_d(K))$ solves the Skorokhod problem $\mathbb{SP}_{L^d}^{R^d}(T_d(S)+X_d)$.
\end{theorem}

\begin{proof}
Clearly, if $L,R$ satisfy Assumption \ref{ass1}, so do $L^d,R^d$. For any $t\geq 0$, it is easy to check that 
\begin{align*}
(H_d(X))_t=X_{d+t}=S_{d+t}+K_{d+t}=(T_d(S))_t+(T_d(K))_t+(S_d+K_d)=(T_d(S))_t+X_d+(T_d(K))_t.
\end{align*}
Moreover, we have
\begin{align*}
L^d(t,(H_d(X))_t)=L(t+d,X_{t+d})\leq 0\leq R(t+d,X_{t+d})=R^d(t,(H_d(X))_t)
\end{align*}
and 
\begin{align*}
&\int_0^\infty I_{\{L^d(s,(H_d(X))_s)<0\}}d(T_d(K^l))_s=\int_d^\infty I_{\{L(s,X_s)<0\}}dK^l_s=0,\\
&\int_0^\infty I_{\{R^d(s,(H_d(X))_s)>0\}}d(T_d(K^r))_s=\int_d^\infty I_{\{R(s,X_s)>0\}}dK^r_s=0.
\end{align*}
The proof is complete.
\end{proof}

\begin{remark}
Theorem \ref{thm3.1} is the extension of Theorem 3.1 in \cite{S1} to the nonlinear reflecting case.
\end{remark}

\subsection{Comparison properties}

In this subsection, we present some comparison properties of the Skorokhod problems with two nonlinear reflecting boundaries. In the proofs, the following inequalities are frequently used:
\begin{align*}
(a+b)^\pm\leq a^\pm+b^\pm, \ (a-b)^\pm\geq a^\pm-b^\pm, \ \textrm{ for any } a,b\in\mathbb{R}.
\end{align*}
Before we investigate the comparison property for double reflected problem, we first establish the comparison property for the single reflected case, which may be of independent interest.
\begin{proposition}\label{lem4.1}
Let Assumption \ref{ass1} (i)-(iii) hold for $R$. Given $c^i_0\in\mathbb{R}$ and $S^i\in D[0,\infty)$ with $S^1_0=S^2_0=0$, $i=1,2$, suppose that there exists a nonnegative $\nu\in I[0,\infty)$ such that $S^2\leq S^1\leq S^2+\nu$. Let $(X^i,K^i)=\mathbb{SP}^R(c_0^i+S^i)$, $i=1,2$. Then, we have
\begin{itemize}
\item[1.] $K_t^1-(c^2_0-c^1_0)^+\leq K^2_t\leq K^1_t+\nu_t+(c_0^1-c_0^2)^+$;
\item[2.] $X_t^2-\nu_t-(c^2_0-c^1_0)^+\leq X^1_t\leq X^2_t+\nu_t+(c_0^1-c_0^2)^+$.
\end{itemize}
\end{proposition}

\begin{proof}
Let $\Psi^i_s$ be such that $R(s,c_0^i+S^i_s+\Psi^i_s)=0$, $i=1,2$, $s\geq 0$. Recalling Remark \ref{single}, we have $K^i_t=\sup_{s\in[0,t]}(\Psi_s^i)^+$. Noting that $S^2\leq S^1$ and
\begin{align*}
R(s,c_0^1+S^1_s+\Psi^1_s)=0=R(s,c_0^2+S^2_s+\Psi^2_s),
\end{align*}
we have 
\begin{align}\label{geq}
\Psi^2_s\geq \Psi^1_s+c_0^1-c_0^2.
\end{align}
Since $S^1\leq S^2+\nu$, it follows that 
\begin{align*}
R(s,c_0^2+S^2_s+\Psi^2_s)=0\leq R(s,c_0^2+S^2_s+\nu_s+\Psi^1_s+c_0^1-c_0^2).
\end{align*}
Consequently, we have
\begin{align}\label{leq}
\Psi^2_s\leq \Psi^1_s+\nu_s+c_0^1-c_0^2.
\end{align}
Recall that $\nu\in I[0,\infty)$ is nonnegative, by Eq. \eqref{leq}, we obtain that 
\begin{align*}
K^2_t=&\sup_{s\in[0,t]}(\Psi^2_s)^+\leq \sup_{s\in[0,t]}(\Psi^1_s+\nu_s+c_0^1-c_0^2)^+\\
\leq &\sup_{s\in[0,t]}(\Psi^1_s+\nu_t+c_0^1-c_0^2)^+\\
\leq &\sup_{s\in[0,t]}(\Psi^1_s)^++\nu_t+(c_0^1-c_0^2)^+\\
=&K^1_t+\nu_t+(c_0^1-c_0^2)^+.
\end{align*} 
Applying Eq. \eqref{geq} yields that 
\begin{align*}
K^1_t=&\sup_{s\in[0,t]}(\Psi^1_s)^+\leq \sup_{s\in[0,t]}(\Psi^2_s+c_0^2-c_0^1)^+\\
\leq &\sup_{s\in[0,t]}(\Psi^2_s)^++(c_0^2-c_0^1)^+=K^2_t+(c_0^2-c_0^1)^+.
\end{align*}
We obtain Property 1. Based on Property 1, together with the facts that $K^2=X^2-c_0^2-S^2$ and $S^1\geq S^2$, it is easy to check that
\begin{align*}
X^1=&c_0^1+S^1+K^1\geq c_0^1+S^1+K^2-\nu-(c^1_0-c^2_0)^+\\
=&c_0^1-c_0^2+S^1-S^2+X^2-\nu-(c^1_0-c^2_0)^+\\
\geq &X^2-\nu-(c^2_0-c^1_0)^+.
\end{align*}
On the other hand, since $S^1\leq S^2+\nu$, we obtain that 
\begin{align*}
X^1=&c_0^1+S^1+K^1\leq c_0^1+S^1+K^2+(c^2_0-c^1_0)^+\\
=&c_0^1-c_0^2+S^1-S^2+X^2+(c^2_0-c^1_0)^+\\
\leq &X^2+\nu+(c^1_0-c^2_0)^+.
\end{align*}
The proof is complete.
\end{proof}

Now we establish the comparison property for the double nonlinear reflected problem.
\begin{proposition}\label{cor4.2}
Let Assumption \ref{ass1} hold. Given $c^i_0\in\mathbb{R}$ and $S^i\in D[0,\infty)$ with $S^1_0=S^2_0=0$, $i=1,2$, suppose that there exists a nonnegative $\nu\in I[0,\infty)$ such that $S^2\leq S^1\leq S^2+\nu$. Let $(X^i,K^i)=\mathbb{SP}_L^R(c_0^i+S^i)$, $i=1,2$. Then, we have
\begin{itemize}
\item[1.] $K_t^1-(c^2_0-c^1_0)^+\leq K^2_t\leq K^1_t+\nu_t+(c_0^1-c_0^2)^+$;
\item[2.] $X_t^2-\nu_t-(c^2_0-c^1_0)^+\leq X^1_t\leq X^2_t+\nu_t+(c_0^1-c_0^2)^+$.
\end{itemize}
\end{proposition}
\begin{proof}
The analysis in the proof of Lemma \ref{lem4.1} implies that 
\begin{align*}
&\Psi^1_s+c_0^1-c_0^2\leq \Psi^2_s\leq \Psi^1_s+\nu_s+c_0^1-c_0^2,\\
&\Phi^1_s+c_0^1-c_0^2\leq \Phi^2_s\leq \Phi^1_s+\nu_s+c_0^1-c_0^2.
\end{align*}
By Theorem \ref{main}, we have
\begin{align*}
K^i_t=\min\left((-(\Phi_0^i)^-)\vee \sup_{r\in[0,t]}\Psi^i_r, \inf_{s\in[0,t]}\left[\Phi^i_s\vee \sup_{r\in[s,t]}\Psi^i_r\right]\right).
\end{align*}
For any $0\leq s\leq t$, it is easy to check that 
\begin{align*}
&\sup_{r\in[s,t]}\Psi^1_r+c^1_0-c^2_0\leq \sup_{r\in[s,t]}\Psi^2_r\leq \sup_{r\in[s,t]}(\Psi^1_r+\nu_r+c_0^1-c^2_0)\leq \sup_{r\in[s,t]}\Psi^1_r+\nu_t+c_0^1-c^2_0,\\
&\Phi^1_s+c_0^1-c_0^2\leq \Phi^2_s\leq \Phi^1_s+\nu_s+c_0^1-c_0^2\leq \Phi^1_s+\nu_t+c_0^1-c_0^2.
\end{align*}
Then, we have
\begin{align*}
&-(\Phi^2_0)^-\leq -(\Phi^1_0+\nu_t+c_0^1-c_0^2)^-\leq -(\Phi_0^1)^-+\nu_t+(c_0^1-c_0^2)^+,\\
&-(\Phi^2_0)^-\geq -(\Phi^1_0+c_0^1-c_0^2)^-\geq -(\Phi^1_0)^--(c_0^1-c_0^2)^-.
\end{align*}
All the above inequalities indicates that 
\begin{align*}
K^1_t-(c^1_0-c^2_0)^-\leq K^2_t\leq K^1_t+\nu_t+(c_0^1-c_0^2)^+.
\end{align*}
Consequently, we have
\begin{align*}
X^2_t-X^1_t=S^2_t-S^1_t+c^2_0-c^1_0+K^2_t-K^1_t\leq c^2_0-c^1_0+\nu_t+(c_0^1-c_0^2)^+=\nu_t+(c_0^1-c_0^2)^-
\end{align*}
and 
\begin{align*}
X^2_t-X^1_t=S^2_t-S^1_t+c^2_0-c^1_0+K^2_t-K^1_t\geq -\nu_t+c^2_0-c^1_0-(c^1_0-c^2_0)^-=-\nu_t-(c^2_0-c^1_0)^-.
\end{align*}
The proof is complete.
\end{proof}

\begin{remark}
Proposition \ref{lem4.1} and Proposition \ref{cor4.2} extend the results of Lemma 4.1 and Corollary 4.2 in \cite{KLR} to the nonlinear reflecting case. Proposition \ref{cor4.2} is also the extension of Proposition 3.4 in \cite{BKR}. However, in Proposition \ref{cor4.2}, we do not need to assume that $S^1=S^2+\nu$ while this condition is proposed in Corollary 4.2 in \cite{KLR} and in Proposition 3.4 in \cite{BKR}. Moreover, suppose $(X^i,K^i)$ solve the Skorokhod problem on $[0,a]$ for $c^i_0+S^i$, $i=1,2$ with fixed $a>0$ (i.e., $(X^i,K^i)=\mathbb{SP}_L^R(c^i_0+S^i)$, with $L(t,x)=x-a$ and $R(t,x)=x$). Corollary 4.2 in \cite{KLR} shows that 
$$K_t^1-2(c^2_0-c^1_0)^+\leq K^2_t\leq K^1_t+2\nu_t+2(c_0^1-c_0^2)^+.$$
Compared with this result, our estimate in Proposition \ref{cor4.2} is more accurate. Actually, Remark 4.3 in \cite{KLR} finally provides the strengthened inequality
$$K_t^1-(c^2_0-c^1_0)^+\leq K^2_t\leq K^1_t+\nu_t+(c_0^1-c_0^2)^+.$$
It is worth pointing out that the proof needs the non-anticipatory properties. Therefore, our proof is  simpler.
\end{remark}

Proposition \ref{cor4.2} only provides the comparison between the net constraining terms $K^1$ and $K^2$. A natural question is that if we could compare the individual constraining terms. The answer is affirmative, which is presented in the following theorem generalizing the result of Theorem 1.7 in \cite{KLR} and Proposition 3.5 in \cite{BKR}.

\begin{theorem}\label{thm1.7}
Let Assumption \ref{ass1} hold. Given $c^i_0\in\mathbb{R}$ and $S^i\in D[0,\infty)$ with $S^1_0=S^2_0=0$, $i=1,2$. Suppose that there exists $\nu\in I[0,\infty)$ such that $S^1=S^2+\nu$. Let $(X^i,K^i)=\mathbb{SP}_L^R(c_0^i+S^i)$ with decomposition $K^i=K^{i,r}-K^{i,l}$, $i=1,2$. Then, we have
\begin{itemize}
\item[1.] $K_t^{1,r}-(c^2_0-c^1_0)^+\leq K^{2,r}_t\leq K^{1,r}_t+\nu_t+(c_0^1-c_0^2)^+$;
\item[2.] $K_t^{2,l}-(c^2_0-c^1_0)^+\leq K^{1,l}_t\leq K^{2,l}_t+\nu_t+(c_0^1-c_0^2)^+$.
\end{itemize}
\end{theorem}

\begin{proof}
Define
\begin{align*}
\alpha:=\inf\{t>0: K^{1,r}_t+\nu_t+(c^1_0-c^2_0)^+<K^{2,r}_t \textrm{ or } K^{1,l}_t+(c_0^2-c_0^1)^+<K^{2,l}_t\}.
\end{align*}
We claim that $\alpha=\infty$. Therefore, for any $t\geq 0$, we have
\begin{align*}
K^{1,r}_t+\nu_t+(c^1_0-c^2_0)^+\geq K^{2,r}_t \textrm{ and } K^{1,l}_t+(c_0^2-c_0^1)^+\geq K^{2,l}_t,
\end{align*}
which are the second inequality in Property 1 and the first inequality in Property 2. We will prove it by way of contradiction. Suppose that $\alpha<\infty$. The proof will be divided into the following parts.

\textbf{Step 1.} We claim that 
\begin{align}\label{eq4.12}
K^{2,r}_\alpha\leq K^{1,r}_\alpha+\nu_\alpha+(c^1_0-c^2_0)^+
\end{align}
and 
\begin{align}\label{eq4.13}
K^{2,l}_\alpha\leq K^{1,l}_\alpha+(c_0^2-c_0^1)^+.
\end{align}

First, by the definition of $\alpha$, for any $s\in[0,\alpha)$, the following two inequalities hold:
\begin{align}\label{eq4.10}
K^{2,r}_s\leq K^{1,r}_s+\nu_s+(c^1_0-c^2_0)^+
\end{align}
and 
\begin{align}\label{eq4.11}
K^{2,l}_s\leq K^{1,l}_s+(c_0^2-c_0^1)^+.
\end{align}
Noting that $\nu, K^{1,r}, K^{1,l}$ are nondecreasing, if $K^{2,r}, K^{2,l}$ are continuous at $\alpha$, \eqref{eq4.12} and \eqref{eq4.13} hold true by using \eqref{eq4.10} and \eqref{eq4.11}, respectively. Now, suppose that $K^{2,r}_\alpha>K^{2,r}_{\alpha-}$. Then we have $K^{2,l}_\alpha=K^{2,l}_{\alpha-}$ and 
\begin{align}\label{eq4.13'}
R(\alpha,X^2_\alpha)=R(\alpha,c_0^2+S^2_\alpha+K^{2,r}_\alpha-K^{2,l}_\alpha)=0.
\end{align}
On the other hand, since $(X^1,K^1)=\mathbb{SP}_L^R(c^1_0+S^1)$, we have
$$
R(\alpha,c^1_0+S^1_\alpha+K^{1,r}_\alpha-K^{1,l}_\alpha)\geq 0.
$$
Combining the above inequality and \eqref{eq4.13'} implies that
\begin{equation}\begin{split}\label{eq4.14}
K^{2,r}_\alpha\leq &c_0^1-c_0^2+S^1_\alpha-S^2_\alpha+K^{2,l}_\alpha-K^{1,l}_\alpha+K^{1,r}_\alpha\\
\leq &c_0^1-c_0^2+\nu_\alpha+K^{2,l}_{\alpha-}-K^{1,l}_{\alpha-}+K^{1,r}_\alpha,
\end{split}\end{equation}
where we have used the facts that $K^{2,l}_\alpha=K^{2,l}_{\alpha-}$, $S^1=S^2+\nu$ and $K^{1,l}$ is nondecreasing. Taking limits as $s\uparrow \alpha$ in \eqref{eq4.11} yields that
\begin{align*}
K^{2,l}_{\alpha-}-K^{1,l}_{\alpha-}\leq (c^2_0-c^1_0)^+.
\end{align*}
Plugging this inequality into Eq. \eqref{eq4.14}, we obtain that 
\begin{align*}
K^{2,r}_\alpha\leq &c_0^1-c_0^2+\nu_\alpha+(c^2_0-c^1_0)^++K^{1,r}_\alpha\\
=&K^{1,r}_\alpha+\nu_\alpha+(c^1_0-c^2_0)^+,
\end{align*}
which is indeed \eqref{eq4.12}. The proof of \eqref{eq4.13} under the assumption that $K^{2,l}_\alpha>K^{2,l}_{\alpha-}$ is similar and is thus omitted.

\textbf{Step 2.} By the definition of $\alpha$ and recalling \eqref{eq4.12} and \eqref{eq4.13}, there exists a sequence $\{s_n\}_{n\in \mathbb{N}}$ converging to $0$ decreasingly such that for any $n\in\mathbb{N}$, one of the following two cases holds
\begin{align}\label{eq4.15}
K^{2,r}_{\alpha+s_n}>K^{1,r}_{\alpha+s_n}+\nu_{\alpha+s_n}+(c^1_0-c^2_0)^+,
\end{align}
\begin{align}\label{eq4.16}
K^{2,l}_{\alpha+s_n}>K^{1,l}_{\alpha+s_n}+(c^2_0-c^1_0)^+.
\end{align}
We claim that \eqref{eq4.15} and \eqref{eq4.16} cannot hold.

Suppose that \eqref{eq4.15} holds. Letting $n$ approach infinity, the right continuity of $K^{2,r}, K^{1,r}$ and  $\nu$ implies that 
\begin{align*}
K^{2,r}_{\alpha}\geq K^{1,r}_{\alpha}+\nu_{\alpha}+(c^1_0-c^2_0)^+.
\end{align*}
The above inequality together with \eqref{eq4.12} yields that 
\begin{align}\label{eq4.17'}
K^{2,r}_{\alpha}=K^{1,r}_{\alpha}+\nu_{\alpha}+(c^1_0-c^2_0)^+.
\end{align}
We first claim that $R(\alpha,X^1_\alpha)=R(\alpha,X^2_\alpha)=0$. In fact, since $K^{1,r}+\nu$ is nondecreasing, it follows from \eqref{eq4.15} and \eqref{eq4.17'} that $K^{2,r}_{\alpha+s_n}>K^{2,r}_\alpha$ for any $n\in\mathbb{N}$. Since $s_n\downarrow 0$, we have $R(\alpha,X^2_\alpha)=0$. Applying \eqref{eq4.13}, \eqref{eq4.17'}, it is easy to check that 
\begin{align*}
0\leq R(\alpha,X^1_\alpha)=&R(\alpha,c^1_0+S^1_\alpha+K^{1,r}_\alpha-K^{1,l}_\alpha)\\
\leq &R(\alpha,c^1_0+S^2_\alpha+\nu_\alpha+K^{2,r}_\alpha-\nu_\alpha-(c_0^1-c_0^2)^+-K^{2,l}_\alpha+(c^2_0-c^1_0)^+)\\
=&R(\alpha,X^2_\alpha)=0.
\end{align*}
Hence, the claim holds true, which implies that $X^1_\alpha=X^2_\alpha$ and $L(\alpha,X^1_\alpha)=L(\alpha,X^2_\alpha)<0$. Since $X^i$, $i=1,2$ are right continuous, there exists some $\varepsilon>0$, such that for any $s\in[0,\varepsilon]$, $L(\alpha+\varepsilon, X^{1}_{\alpha+\varepsilon})<0$ and $L(\alpha+\varepsilon, X^{2}_{\alpha+\varepsilon})<0$. Thus, for any $s\in[0,\varepsilon]$, $K^{i,l}_{\alpha+s}=K^{i,l}_\alpha$, $i=1,2$. Therefore, $(X^i_{\alpha+s},K^{i,r}_{\alpha+s}-K^{i,r}_\alpha)_{s\in[0,\varepsilon]}$ can be seen as the solution to the  Skorokhod problem with one constraint $\mathbb{SP}^{R^\alpha}(S^{i,\alpha})$ on the time interval $[0,\varepsilon]$, where 
$$R^\alpha(t,x):=R(\alpha+t,x), \ S^{i,\alpha}_t:=X^i_\alpha+S^{i}_{\alpha+t}-S^{i}_\alpha.$$
Applying Property 1 of Proposition \ref{lem4.1} and noting that $X^1_\alpha=X^2_\alpha$, for any $s\in[0,\varepsilon]$, we have
\begin{align*}
K^{2,r}_{\alpha+s}-K^{2,r}_\alpha\leq K^{1,r}_{\alpha+s}-K^{1,r}_\alpha+\nu_{\alpha+s}-\nu_\alpha.
\end{align*}
Plugging \eqref{eq4.17'} into the above equation implies that 
\begin{align*}
K^{2,r}_{\alpha+s}\leq K^{1,r}_{\alpha+s}+\nu_{\alpha+s}+(c^1_0-c^2_0)^+,
\end{align*}
which contradicts \eqref{eq4.15}. Thus, \eqref{eq4.15} cannot hold.

It remains to show \eqref{eq4.16} does not hold. By the above analysis, together with \eqref{eq4.10} and \eqref{eq4.12}, there exists some $\delta>0$, such that for any $s\in[0,\alpha+\delta]$,
\begin{align}\label{eq4.17}
K^{2,r}_{s}\leq K^{1,r}_{s}+\nu_{s}+(c^1_0-c^2_0)^+.
\end{align}
Recalling Remark \ref{rem4.9}, $(-X^i,K^{i,l})$ may be interpreted as the solution to the Skorokhod problem with one constraint $\mathbb{SP}^{\tilde{L}}(-c^i_0-S^i-K^{i,r})$, $i=1,2$, where $\tilde{L}(t,x)=-L(t,-x)$. For any $s\geq 0$, let $\Phi^{i.l}_s$ be the solution to the following equation
$$
\tilde{L}(s,-c^i_0-S^i_s-K^{i,r}_s+\Phi^{i,l}_s)=0.
$$
Then, for any $t\in[0,\alpha+\delta]$, we have $c^1_0+S^1_t+K^{1,r}_t-\Phi^{1,l}_t=c^2_0+S^2_t+K^{2,r}_t-\Phi^{2,l}_t$ and 
\begin{align*}
K^{2,l}_t=&\sup_{s\in[0,t]}(\Phi^{2,l}_s)^+=\sup_{s\in[0,t]}(c^2_0+S^2_s+K^{2,r}_s-c^1_0-S^1_s-K^{1,r}_s+\Phi^{1,l}_s)^+\\
\leq &\sup_{s\in[0,t]}(c^2_0-c^1_0-\nu_s+\nu_s+(c^1_0-c^2_0)^++\Phi^{1,l}_s)^+\\
\leq &\sup_{s\in[0,t]}(\Phi^{1,l}_s)^++(c^2_0-c^1_0)^+=K^{1,l}_t+(c^2_0-c^1_0)^+,
\end{align*}
where we have used \eqref{eq4.17} and the fact that $S^2=S^1+\nu$. However, the above inequality contradicts \eqref{eq4.16}. All the above analysis indicates that neither \eqref{eq4.15} nor \eqref{eq4.16} holds, which means that $\alpha=\infty$. That is, the second inequality in property 1 and the first inequality in property 2 are satisfied.

Now, set 
\begin{align*}
\beta:=\inf\{t>0: K^{2,l}_t+\nu_t+(c_0^1-c_0^2)^+<K^{1,l}_t \textrm{ or } K^{2,r}_t+(c_0^2-c_0^1)^+<K^{1,r}_t\}.
\end{align*}
By a similar analysis as above, we may show that $\beta=\infty$. This implies that the first inequality in property 1 and the second inequality in property 2 are satisfied. The proof is complete.
\end{proof}

\begin{remark}
Applying the comparison properties Proposition \ref{cor4.2} and Theorem \ref{thm1.7}, we can also show that the solution to the Skorokhod problem $\mathbb{SP}_L^R(S)$ is unique.
\end{remark}

All the above results in this subsection show the comparison properties of solutions to Skorokhod problems with respect to the input function $S$. In the following, we provide the monotonicity property of the individual constraining functions with respect to the nonlinear reflecting boundaries $L,R$.

\begin{lemma}\label{lemma3.1}
Suppose $(L^i,R^i)$ satisfy Assumption \ref{ass1}, $i=1,2$ with $R^1\equiv R^2$ and $L^1\leq L^2$. For any given $S\in D[0,\infty)$, let $(X^i,K^i)$ be the solution to the Skorokhod problem $\mathbb{SP}_{L^i}^{R^i}(S)$ with $K^i=K^{i,r}-K^{i,l}$. Then, for any $t\geq 0$, we have
\begin{align*}
K^{2,r}_t\geq K^{1,r}_t \textrm{ and } K^{2,l}_t\geq K^{1,l}_t.
\end{align*}
\end{lemma}

\begin{proof}
 For any $t\geq 0$, let $\Phi^i_t$, $\Psi_t^i$, $i=1,2$ be the solutions to the following equations, respectively
\begin{align*}
L^i(t,S_t+\Phi^i_t)=0, \ R^i(t,S_t+\Psi^i_t)=0.
\end{align*}
It is easy to check that $\Psi^1_t=\Psi_t^2$ and $\Phi^1_t\geq \Phi^2_t$. By Theorem \ref{main}, we have $X^i_t=S_t+K^i_t$, $i=1,2$, where
\begin{align*}
K^i_t=\min\left((-(\Phi_0^i)^-)\vee \sup_{r\in[0,t]}\Psi^i_r, \inf_{s\in[0,t]}\left[\Phi^i_s\vee \sup_{r\in[s,t]}\Psi^i_r\right]\right).
\end{align*}
Then, we obtain that $K^1_t\geq K^2_t$ and 
\begin{align}\label{equa3.4}
X^1_t\geq X^2_t.
\end{align}
 Note that $K^{i,l}_t=-K^i_t+K^{i,r}_t$. It suffices to prove that for any $t\geq 0$,
\begin{align}\label{equa3.5}
K^{2,r}_t\geq K^{1,r}_t.
\end{align} 

At time $0$, if $R^1(0,S_0)=R^2(0,S_0)\geq 0$, then we have $K^{1,r}_0=K^{2,r}_0=0$. If $R^1(0,S_0)=R^2(0,S_0)< 0$, then we have $K^{1,r}_0=K^{2,r}_0=\Psi^1_0>0$. Hence, \eqref{equa3.5} holds at the initial time.

Now, set 
\begin{align*}
t^*:=\inf\{s\geq 0: K^{2,r}_s<K^{1,r}_s\}.
\end{align*}
We claim that $t^*=\infty$, which completes the proof. We will prove this assertion by contradiction. Suppose that $t^*<\infty$. Then, we have
\begin{align}\label{equa3.6}
K^{2,r}_{t^*-}\geq K^{1,r}_{t^*-}
\end{align}
and for any $\varepsilon_0>0$, there exists $\varepsilon\in(0,\varepsilon_0)$ such that 
\begin{align}\label{equa3.7}
K^{2,r}_{t^*+\varepsilon}< K^{1,r}_{t^*+\varepsilon}.
\end{align}

We first show that $K^{2,r}_{t^*}\geq K^{1,r}_{t^*}$. If $K^{1,r}_{t^*}-K^{1,r}_{t^*-}=0$, it is clear that $K^{2,r}_{t^*}-K^{2,r}_{t^*-}\geq K^{1,r}_{t^*}-K^{1,r}_{t^*-}$ since $K^{2,r}$ is nondecreasing. For the case that $K^{1,r}_{t^*}-K^{1,r}_{t^*-}>0$, we have $R^1(t^*,X^1_{t^*})=0$. The facts that $R^1\equiv R^2$ and $(X^2,K^2)=\mathbb{SP}_{L^2}^{R^2}(S)$ imply that
\begin{align*}
R^2(t^*,X^1_{t^*})=0\leq R^2(t^*,X^2_{t^*}),
\end{align*}
which together with \eqref{equa3.4} indicates that $X^2_{t^*}= X^1_{t^*}$. Recalling Remark \ref{remark}, we obtain that 
\begin{align*}
K^{2,r}_{t^*}-K^{2,r}_{t^*-}=&(X^{2}_{t^*}-X^{2}_{t^*-}-(S_{t^*}-S_{t^*-}))^+\\
\geq &(X^{1}_{t^*}-X^{1}_{t^*-}-(S_{t^*}-S_{t^*-}))^+\\
=&K^{1,r}_{t^*}-K^{1,r}_{t^*-},
\end{align*}
where we have used \eqref{equa3.4} in the second inequality. Therefore, the inequality $K^{2,r}_{t^*}-K^{2,r}_{t^*-}\geq K^{1,r}_{t^*}-K^{1,r}_{t^*-}$ always holds true. Combining with \eqref{equa3.6} yields that $K^{2,r}_{t^*}\geq K^{1,r}_{t^*}$. Then, we consider the following two cases.

\textbf{Case 1. } $K^{2,r}_{t^*}> K^{1,r}_{t^*}$. Due to the right-continuity of $K^{1,r}$ and $K^{2,r}$, for $\varepsilon>0$ small enough, we have $K^{2,r}_{t^*+\varepsilon}> K^{1,r}_{t^*+\varepsilon}$, which contradicts \eqref{equa3.7}.

\textbf{Case 1. } $K^{2,r}_{t^*}=K^{1,r}_{t^*}$. Noting that $K^{i,r}$ are nondecreasing, by \eqref{equa3.7}, for any $\varepsilon_0>0$, there exists $\varepsilon\in(0,\varepsilon_0)$ such that $K^{1,r}_{t^*}<K^{1,r}_{t^*+\varepsilon}$. According to Definition \ref{def}, this implies that $R^1(t^*,X^1_{t^*})=0$. Recalling \eqref{equa3.4} and the fact that $R^1\equiv R^2$, it follows that 
$$0=R^2(t^*,X^1_{t^*})\geq R^2(t^*,X^2_{t^*})\geq 0,$$
 which indicates that $X^1_{t^*}=X^2_{t^*}$ and $R^1(t^*,X^1_{t^*})=R^{2}(t^*,X^2_{t^*})=0$. Consequently, we have 
 $$L^1(t^*,X^1_{t^*})<0 \textrm{ and } L^{2}(t^*,X^2_{t^*})<0.$$
  Due to the right-continuity of $X^i$, there exists some $\delta>0$ small enough, such that for any $t\in[t^*,t^*+\delta]$, 
\begin{align*}
L^1(t,X^1_t)<0, \ L^2(t,X^2_t)<0.
\end{align*}
Therefore, for any $t\in[t^*,t^*+\delta]$, we have $K^{i,l}_t=K^{i,l}_{t^*}$ and thus $K^i_t-K^i_{t^*}=K^{i,r}_t-K^{i,r}_{t^*}$, $i=1,2$. Then, similar as Theorem \ref{thm3.1}, we deduce that on the time interval $[0,\delta]$, 
$$(H_{t^*}(X^i),T_{t^*}(K^{i,r}))=\mathbb{SP}^{R^{i,t^*}}(X^i_{t^*}+T_{t^*}(S)),\  i=1,2,$$
 where $R^{i,t^*}(t,x)=R^i(t+t^*,x)$ and $T_{t^*},H_{t^*}$ are defined in \eqref{tdhd}.  For any $s\in[0,\delta]$, let $\Psi^{i,t^*}_s$, $i=1,2$, be the solution to the following equation
\begin{align*}
R^{i,t^*}(s,X^i_{t^*}+(T_{t^*}(S))_s+\Psi^{i,t^*}_s)=0.
\end{align*}
Since $R^1\equiv R^2$ and $X^1_{t^*}=X^2_{t^*}$, it follows that $\Psi^{1,t^*}_s=\Psi^{2,t^*}_s$. By Remark \ref{single}, we have 
\begin{align*}
K^{1,r}_{t^*+\delta}-K^{1,r}_{t^*}=(T_{t^*}(K^{1,r}))_\delta=\sup_{s\in[0,\delta]}(\Psi^{1,t^*}_s)^+=\sup_{s\in[0,\delta]}(\Psi^{2,t^*}_s)^+=(T_{t^*}(K^{2,r}))_\delta=K^{2,r}_{t^*+\delta}-K^{2,r}_{t^*}.
\end{align*} 
Since we are considering the case $K^{2,r}_{t^*}=K^{1,r}_{t^*}$, we deduce that $K^{2,r}_{t^*+\delta}=K^{1,r}_{t^*+\delta}$, which contradicts \eqref{equa3.7}.

Therefore, all the above analysis indicates that $t^*=\infty$, and the desired result holds true.
\end{proof}

\begin{remark}
Suppose that $l,\tilde{l},r,\tilde{r}\in D[0,\infty)$ with $l\equiv \tilde{l}$, $r\leq \tilde{r}$ and $\inf_{t\geq 0}(r_t-l_t)>0$. Let $R^1(t,x)=x-\tilde{l}_t$, $R^2(t,x)=x-l_t$, $L^1(t,x)=x-\tilde{r}_t$ and $L^2(t,x)=x-r_t$. Clearly, $R^i$, $L^i$, $i=1,2$ satisfy the assumptions in Lemma \ref{lemma3.1}. In this case, our result degenerates to the result in Lemma 3.1 in \cite{BKR}.
\end{remark}

By a similar analysis as the proof of Lemma \ref{lemma3.1}, we obtain the following lemma.

\begin{lemma}\label{lemma3.2}
Suppose $(L^i,R^i)$ satisfy Assumption \ref{ass1}, $i=1,2$ with $R^1\geq R^2$ and $L^1\equiv L^2$. For any given $S\in D[0,\infty)$, let $(X^i,K^i)$ be the solution to the Skorokhod problem $\mathbb{SP}_{L^i}^{R^i}(S)$ with $K^i=K^{i,r}-K^{i,l}$. Then, for any $t\geq 0$, we have
\begin{align*}
K^{2,r}_t\geq K^{1,r}_t \textrm{ and } K^{2,l}_t\geq K^{1,l}_t.
\end{align*}
\end{lemma}


\begin{proposition}\label{proposition3.3}
Suppose $(L^i,R^i)$ satisfy Assumption \ref{ass1}, $i=1,2$ with $R^1\geq R^2$ and $L^1\leq L^2$. For any given $S\in D[0,\infty)$, let $(X^i,K^i)$ be the solution to the Skorokhod problem $\mathbb{SP}_{L^i}^{R^i}(S)$ with $K^i=K^{i,r}-K^{i,l}$. Then, for any $t\geq 0$, we have
\begin{align*}
K^{2,r}_t\geq K^{1,r}_t \textrm{ and } K^{2,l}_t\geq K^{1,l}_t.
\end{align*}
\end{proposition}

\begin{proof}
Let $(X^*,K^*)$ be the solution to the Skorokhod problem $\mathbb{SP}_{L^1}^{R^2}(S)$ with $K^*=K^{*,r}-K^{*,l}$. By Lemma \ref{lemma3.1}, we have 
\begin{align*}
K^{2,r}_t\geq K^{*,r}_t \textrm{ and } K^{2,l}_t\geq K^{*,l}_t.
\end{align*}
By Lemma \ref{lemma3.2}, we have
\begin{align*}
K^{*,r}_t\geq K^{1,r}_t \textrm{ and } K^{*,l}_t\geq K^{1,l}_t.
\end{align*}
The above inequalities yield the desired result.
\end{proof}

\begin{remark}
Proposition \ref{proposition3.3} is the extension of Proposition 3.3 in \cite{BKR} to the case of two nonlinear reflecting boundaries.
\end{remark}

\subsection{Continuity properties}

In this subsection, we discuss the continuity properties for the Skorokhod problems with two nonlinear reflecting boundaries under the uniform metric and $J_1$ metric $d_0$, respectively. The proof of these properties is based on the representation for $K$ obtained in \eqref{K}. Therefore, we first establish some estimates for $\Phi$ and $\Psi$.

\begin{lemma}\label{continuity in S}
Suppose that $(L^i,R^i)$ satisfy Assumption \ref{ass1}, $i=1,2$. Given $S^i\in D[0,\infty)$, $i=1,2$, for any $t\geq 0$, let $\Phi^i,\Psi^i$, $i=1,2$ be the solutions to the following equations
\begin{align*}
L^i(t,S^i_t+\Phi^i_t)=0, \ R^i(t,S^i_t+\Psi^i_t)=0.
\end{align*}
Then, we have 
\begin{align*}
&|\Phi^1_t-\Phi^{2}_t|\leq  \frac{C}{c}|S_t-S'_t|+\frac{1}{c}\sup_{x\in \mathbb{R}}|L^1(t,x)-L^2(t,x)|, \\
& |\Psi^1_t-\Psi^{2}_t|\leq  \frac{C}{c}|S_t-S'_t|+\frac{1}{c}\sup_{x\in \mathbb{R}}|R^1(t,x)-R^2(t,x)|.
\end{align*}
\end{lemma}

\begin{proof}
It suffices to prove the first inequality. Simple calculation yields that 
\begin{align*}
c|\Phi^1_t-\Phi^{2}_t|\leq &|L^1(t,S^1_t+\Phi^1_t)-L^1(t,S^1_t+\Phi^{2}_t)|=|L^2(t,S^2_t+\Phi^{2}_t)-L^1(t,S^1_t+\Phi^{2}_t)|\\
\leq &|L^2(t,S^2_t+\Phi^{2}_t)-L^2(t,S^1_t+\Phi^{2}_t)|+|L^2(t,S^1_t+\Phi^{2}_t)-L^1(t,S^1_t+\Phi^{2}_t)|\\
\leq &C|S_t^1-S_t^2|+\sup_{x\in\mathbb{R}}|L^1(t,x)-L^2(t,x)|.
\end{align*}
The proof is complete.
\end{proof}

\begin{proposition}\label{prop4.1}
Suppose that $(L^i,R^i)$ satisfy Assumption \ref{ass1}, $i=1,2$. Given $S^i\in D[0,\infty)$, let $(X^i,K^i)$ be the solution to the nonlinear Skorokhod problem $\mathbb{SP}_{L^i}^{R^i}(S^i)$. Then, we have
\begin{align*}
\sup_{t\in[0,T]}|K^1_t-K^2_t|\leq \frac{C}{c}\sup_{t\in[0,T]}|S^1_t-S^2_t|+\frac{1}{c}
(\bar{L}_T\vee\bar{R}_T),
\end{align*}
where
\begin{align*}
&\bar{L}_T:=\sup_{(t,x)\in[0,T]\times\mathbb{R}}|L^1(t,x)-L^2(t,x)|,\\
&\bar{R}_T:= \sup_{(t,x)\in[0,T]\times\mathbb{R}}|R^1(t,x)-R^2(t,x)|.
\end{align*}
\end{proposition}

\begin{proof}
Note that for any $x_i,y_i\in\mathbb{R}$, $i=1,2$, the following inequalities hold:
\begin{align*}
|x_1\wedge x_2-y_1\wedge y_2|&\leq |x_1-y_1|\vee |x_2-y_2|,\\
|x_1^+-x_2^+|&\leq |x_1-x_2|.
\end{align*}
It is easy to check that
\begin{align*}
&|(-\Phi^1_0)^+\wedge \inf_{r\in[0,t]}(-\Psi^1_r)-(-\Phi^2_0)^+\wedge \inf_{r\in[0,t]}(-\Psi^2_r)|\\
\leq &|(-\Phi^1_0)^+-(-\Phi^2_0)^+|\vee| \inf_{r\in[0,t]}(-\Psi^1_r)- \inf_{r\in[0,t]}(-\Psi^2_r)|\\
\leq &|\Phi^1_0-\Phi^2_0|\vee \sup_{r\in[0,t]} |\Psi^1_r-\Psi^2_r|
\end{align*}
and
\begin{align*}
&|\sup_{s\in[0,t]}[(-\Phi^1_s)\wedge \inf_{r\in[s,t]}(-\Psi^1_r)]-\sup_{s\in[0,t]}[(-\Phi^2_s)\wedge \inf_{r\in[s,t]}(-\Psi^2_r)]|\\
\leq &\sup_{s\in[0,t]}|(-\Phi^1_s)\wedge \inf_{r\in[s,t]}(-\Psi^1_r)-(-\Phi^2_s)\wedge \inf_{r\in[s,t]}(-\Psi^2_r)|\\
\leq &\sup_{s\in[0,t]}\left[|\Phi^1_s-\Phi^2_s|\vee |\inf_{r\in[s,t]}(-\Psi^1_r)-\inf_{r\in[s,t]}(-\Psi^2_r)|\right]\\
\leq &\sup_{s\in[0,t]}\left[|\Phi^1_s-\Phi^2_s|\vee \sup_{r\in[s,t]}|\Psi^1_r-\Psi^2_r|\right]\\
\leq &\left[\sup_{s\in[0,t]}|\Phi^1_s-\Phi^2_s|\vee \sup_{s\in[0,t]}|\Psi^1_s-\Psi^2_s|\right].
\end{align*}
Recalling the construction of $K^i$ in \eqref{K} and  applying Lemma \ref{continuity in S}, we obtain the desired result. 
\end{proof}

\begin{remark}
Proposition \ref{prop4.1} is the generalization of Corollary 1.6 in \cite{KLR} for the uniform metric and Proposition 4.1 in \cite{S1}. In fact, for any given $\alpha^i,\beta^i\in D[0,\infty)$ with $\inf_{t}(\beta^i_t-\alpha^i_t)>0$, $i=1,2$, let $L^i(t,x)=x-\beta^i_t$ and $R^i(t,x)=x-\alpha^i_t$. The result in Proposition \ref{prop4.1} coincides with (4.2) in \cite{S1}.
\end{remark}

For any fixed $T>0$, let $\mathcal{M}_T$ be the collection of  strictly increasing continuous functions $\lambda$ of $[0,T]$ onto itself. The $J_1$ metric $d_{0,T}$ on $D[0,T]$ is defined by 
\begin{align*}
d_{0,T}(f,g)=\inf_{\lambda\in\mathcal{M}_T}(\sup_{t\in[0,T]}|\lambda(t)-t|\vee \sup_{t\in[0,T]}|f_t-g_{\lambda(t)}|).
\end{align*}

\begin{proposition}\label{prop4.2}
Suppose that Assumption \ref{ass1} holds. Given $S,S'\in D[0,\infty)$, let $(X,K)$, $(X',K')$ be the solution to the nonlinear Skorokhod problem $\mathbb{SP}_L^R(S)$ and $\mathbb{SP}_L^R(S')$, respectively. Then, for any $T>0$, we have
\begin{align*}
d_{0,T}(K,K')\leq \frac{1}{c}(\hat{L}_T\vee\hat{R}_T)
+\frac{C}{c}d_{0,T}(S,S'),
\end{align*}
where 
\begin{align*}
&\hat{L}_T:=\sup_{(t,s,x)\in[0,T]\times[0,T]\times\mathbb{R}}|L(t,x)-L(s,x)|,\\
& \hat{R}_T:=\sup_{(t,s,x)\in[0,T]\times[0,T]\times\mathbb{R}}|R(t,x)-R(s,x)|.
\end{align*}
\end{proposition}

\begin{proof}
Without loss of generality, we assume that $S\neq S'$.  By the definition of $d_{0,T}$, for any $\delta>0$, there exists some $\lambda \in\mathcal{M}_T$ such that 
\begin{align*}
\sup_{t\in[0,T]}|\lambda(t)-t|&\leq d_{0,T}(S,S')+\delta [1\wedge d_{0,T}(S,S')],\\
\sup_{t\in[0,T]}|S'_t-S_{\lambda(t)}|&\leq d_{0,T}(S,S')+\delta [1\wedge d_{0,T}(S,S')].
\end{align*}

Given $\lambda \in\mathcal{M}_T$, for any $t\in[0,T]$, applying the definition of $K$ in \eqref{K}, it is easy to check that 
\begin{align*}
K_{\lambda(t)}=-\max\left((-\Phi_0)^+\wedge \inf_{r\in[0,t]}(-\Psi_{\lambda(r)}),\sup_{s\in[0,t]}\left[(-\Phi_{\lambda(s)})\wedge \inf_{r\in[s,t]}(-\Psi_{\lambda(r)})\right] \right).
\end{align*}
That is, $(X\circ \lambda,K\circ\lambda)$ is the solution to the Skorokhod problem $\mathbb{SP}_{L\circ\lambda}^{R\circ\lambda}(S\circ \lambda)$ on $[0,T]$, where $(f\circ\lambda)_t=f_{\lambda(t)}$ for $f=X,K,S$ and $(g\circ\lambda)(t,x)=g(\lambda(t),x)$ for $g=L,R$. By Proposition \ref{prop4.1}, we have
\begin{align*}
\sup_{t\in[0,T]}|K'_t-K_{\lambda(t)}|\leq &\frac{1}{c}\left[\sup_{(t,x)\in[0,T]\times\mathbb{R}}|L(t,x)-L({\lambda(t)},x)|\vee \sup_{(t,x)\in[0,T]\times\mathbb{R}}|R(t,x)-R({\lambda(t)},x)|\right]\\
&+\frac{C}{c}\sup_{t\in[0,T]}|S'_t-S_{\lambda(t)}|\\
\leq &\frac{1}{c}(\hat{L}_T\vee\hat{R}_T)
+\frac{C}{c}\left(d_{0,T}(S,S')+\delta [1\wedge d_{0,T}(S,S')]\right).
\end{align*}
Since the above inequality holds for any $\delta>0$, we obtain the desired result.
\end{proof}

\begin{remark}
Proposition \ref{prop4.2} is the generalization of Corollary 1.6 in \cite{KLR} for the $J_1$ metric and Proposition 4.2 in \cite{S1}. Different from the results in \cite{KLR} for the Skorokhod map on time-independent interval $[0,a]$, where $a$ is a positive constant, the solution to the Skorokhod problem with two nonlinear reflecting boundaries is not continuous under $J_1$ metric. The terms $\hat{L}_T$, $\hat{R}_T$ may be regarded as the ``oscillations" of $L$ and $R$, which cannot be omitted (see Example 4.1 in \cite{S1}).
\end{remark}






\end{document}